\documentclass[11pt,oneside]{amsart}
\PassOptionsToPackage{nameinlink}{cleveref}
\usepackage[T1]{fontenc}
\UseRawInputEncoding
\usepackage{color}
\usepackage{colortbl}
\usepackage{array}
\usepackage{longtable}
\usepackage{booktabs}
\usepackage{mathtools}
\usepackage{amsbsy}
\usepackage{amstext}
\usepackage{amsthm}
\usepackage{amssymb}
\usepackage{geometry}
\usepackage[bookmarks=true,bookmarksnumbered=false,bookmarksopen=false,
 breaklinks=false,pdfborder={0 0 1},backref=false,colorlinks=true]
 {hyperref}
\hypersetup{
 allcolors={[rgb]{0,0.2,0.6}}}

\makeatletter

\providecommand{\tabularnewline}{\\}

\numberwithin{equation}{section}
\numberwithin{figure}{section}

\usepackage{listings}
\usepackage{shuffle}

\usepackage{crossreftools}

\AddToHook{begindocument/end}{%
  \crefname{thm}{Theorem}{Theorems}%
  \Crefname{thm}{Theorem}{Theorems}%

  \crefname{lem}{Lemma}{Lemmas}%
  \Crefname{lem}{Lemma}{Lemmas}%
}

\makeatother

\theoremstyle{plain}
\newtheorem{thm}{\protect\theoremname}
\newtheorem{cor}[thm]{\protect\corollaryname}
\theoremstyle{definition}
\newtheorem{example}[thm]{\protect\examplename}
\newtheorem{defn}[thm]{\protect\definitionname}
\theoremstyle{plain}
\newtheorem{lem}[thm]{\protect\lemmaname}
\newtheorem{prop}[thm]{\protect\propositionname}
\newtheorem{conjecture}[thm]{\protect\conjecturename}
\theoremstyle{remark}
\newtheorem{rem}[thm]{\protect\remarkname}
\theoremstyle{plain}
\newtheorem*{thm*}{\protect\theoremname}
\newtheorem*{lem*}{\protect\lemmaname}

\usepackage{cleveref}
\AddToHook{env/cor/begin}{\crefalias{thm}{cor}}
\AddToHook{env/example/begin}{\crefalias{thm}{example}}
\AddToHook{env/defn/begin}{\crefalias{thm}{defn}}
\AddToHook{env/lem/begin}{\crefalias{thm}{lem}}
\AddToHook{env/prop/begin}{\crefalias{thm}{prop}}
\AddToHook{env/conjecture/begin}{\crefalias{thm}{conjecture}}
\AddToHook{env/rem/begin}{\crefalias{thm}{rem}}
\crefname{conjecture}{conjecture}{conjectures}
\Crefname{conjecture}{Conjecture}{Conjectures}
\crefname{cor}{corollary}{corollaries}
\Crefname{cor}{Corollary}{Corollaries}
\crefname{defn}{definition}{definitions}
\Crefname{defn}{Definition}{Definitions}
\crefname{example}{example}{examples}
\Crefname{example}{Example}{Examples}
\crefname{lem}{lemma}{lemmas}
\Crefname{lem}{Lemma}{Lemmas}
\crefname{prop}{proposition}{propositions}
\Crefname{prop}{Proposition}{Propositions}
\crefname{rem}{remark}{remarks}
\Crefname{rem}{Remark}{Remarks}
\crefname{thm}{theorem}{theorems}
\Crefname{thm}{Theorem}{Theorems}
\providecommand{\conjecturename}{Conjecture}
\providecommand{\corollaryname}{Corollary}
\providecommand{\definitionname}{Definition}
\providecommand{\examplename}{Example}
\providecommand{\lemmaname}{Lemma}
\providecommand{\propositionname}{Proposition}
\providecommand{\remarkname}{Remark}
\providecommand{\theoremname}{Theorem}

\begin{document}
\address[Minoru Hirose]{Graduate School of Science and Engineering, Kagoshima University, 1-21-35 Korimoto, Kagoshima, Kagoshima 890-0065, Japan}
\email{hirose@sci.kagoshima-u.ac.jp}
\title{Confluence relations for $q$-analogues of multiple zeta values}
\author{Minoru Hirose}
\begin{abstract}
In this paper, we construct confluence relations for $q$-analogues
of multiple zeta values by adapting the classical construction to
the $q$-setting. Our construction uses $q$-analogues of multiple
polylogarithms depending on an auxiliary variable $z$, functional
relations arising from their $q$-differential equations, and the
(regularized) limit as $z$ tends to $1$. We prove that the resulting
family of relations contains the stuffle product relations and, assuming
a certain conjectural identity, the duality relations for both the
Bradley--Zhao and Schlesinger--Zudilin models.
\end{abstract}

\maketitle

\section{Introduction}

Understanding the relations among multiple zeta values is one of the
central problems in their study. The analogous problem for $q$-analogues
of multiple zeta values ($q$MZVs) is likewise an important theme.
Although several families of $q$-series can be used as generators
for the space of $q$MZVs (see, e.g., \cite{Bachmann_Kuehn_dimension_qMZV,Brindle_unified_qMZV,Burmester_balanced_qMZV,Singer_qMZV}),
in this paper we work with the $q$MZVs of the Schlesinger--Zudilin
(SZ) model \cite{Schlesinger_qMPL,Zudilin_MZV_relations}. The algebraic
relations and duality structures associated with this and related
models have been investigated, for example, in \cite{SZ_stuffle_duality,EbrahimiFard_Manchon_Singer_duality,Zhao_qMZV_duality}.

The confluence relations for classical multiple zeta values introduced
in \cite{Hirose_Sato_confluence} were originally formulated in terms
of iterated integrals with a moving puncture. We describe their basic
mechanism here in terms of multiple polylogarithms (MPLs). Note that
the auxiliary variable $z$ used in the present paper is the inverse
of the variable denoted by $z$ in \cite{Hirose_Sato_confluence}.
With this convention, one considers MPLs depending on $z$ and uses
their differential equations to express them as linear combinations
of MPLs of the form
\[
\operatorname{Li}_{k_{1},\ldots,k_{d}}(1,\ldots,1,z)
\]
with coefficients given by MZVs. Letting $z$ tend to $1$, with a
suitable regularization when necessary, then yields relations among
MZVs.

In this paper, we construct a $q$-analogue of this procedure. The
MPLs are replaced by $q$-analogues of multiple polylogarithms (qMPLs),
and their differential equations by $q$-differential equations. In
the same way, the relevant qMPLs are expressed in a corresponding
standard form with coefficients given by qMZVs, and letting $z$ tend
to $1$, again with regularization when necessary, yields relations
among qMZVs. We call the relations obtained in this way confluence
relations for qMZVs.

We further investigate how the confluence relations are related to
known families of relations among $q$MZVs. We prove that the stuffle
product relations are contained in the confluence relations. We also
formulate a certain conjectural identity concerning Yamamoto's $q$-duality
\cite{Yamamoto_qMPL_duality}. Assuming this identity, we show that
the duality relations for both the Bradley--Zhao (BZ) and SZ models
are contained in $I_{\mathrm{CF}}$. It is conjectured that all $\mathbb{Q}$-linear
relations among SZ $q$MZVs are obtained by the stuffle product and
SZ duality (cf. \cite[Conjecture 1.1]{SZ_stuffle_duality}). Hence,
these results provide strong evidence that the confluence relations
give all $\mathbb{Q}$-linear relations among SZ $q$MZVs.

The remainder of this paper is organized as follows. In \Cref{sec:basic_idea},
we explain the basic mechanism of confluence relations, first in the
classical setting and then in the $q$-setting. In \Cref{sec:qMZVs},
we introduce the SZ $q$MZVs and their stuffle product, and establish
a residue description of stuffle regularization. In \Cref{sec:Review_qMPL},
we review $q$MPLs and their $q$-difference formulas. In \Cref{sec:formulation_conf},
we formulate the confluence relations using formal $q$MPLs. In \Cref{sec:Conf_and_Stuffle},
we prove that the stuffle product relations are contained in the confluence
relations. In \Cref{sec:Conf_and_Duality}, we formulate a conjectural
identity and show that it implies the inclusion of the BZ and SZ duality
relations in $I_{\mathrm{CF}}$. In \Cref{sec:Further_questions},
we discuss several questions arising from our construction. \Cref{app:notation-summary}
summarizes the notation for the spaces of formal $q$MPLs, \Cref{app:Proof_of_some_claims}
contains proofs of several technical claims, and \Cref{app:Tables}
presents the table related to the conjecture stated in \Cref{sec:Conf_and_Duality}.

\subsection*{Notation}

For a subset $X\subset\mathbb{Z}$, we denote by $q^{X}$ the set
$\{q^{n}\,\mid\,n\in X\}$. We denote by $\mathcal{M}(\mathbb{C})$
the field of meromorphic functions in $z$ on $\mathbb{C}$.

Throughout this paper, we fix a complex parameter $q$ with $0<|q|<1$.
We define the $q$-differential operator
\[
D_{q}\colon\mathcal{M}(\mathbb{C})\longrightarrow\mathcal{M}(\mathbb{C})
\]
by
\[
(D_{q}f)(z)\coloneqq\frac{f(z)-f(zq)}{z}.
\]
For a meromorphic function $f(z)\in\mathcal{M}(\mathbb{C})$ having
no pole at $z=0$, we define the Jackson integral by
\[
\int f(z)\,d_{q}z\coloneqq\sum^{\infty}_{n=0}zq^{n}f(zq^{n}).
\]
Then we have
\[
D_{q}\left(\int f(z)\,d_{q}z\right)=f(z),\qquad\int(D_{q}f)(z)\,d_{q}z=f(z)-f(0)
\]
for every meromorphic function $f(z)\in\mathcal{M}(\mathbb{C})$ having
no pole at $z=0$.

\subsection*{Acknowledgements}

The author thanks Hidekazu Furusho for helpful comments on the manuscript.
This work was supported by JSPS KAKENHI Grant Number JP22K03244.

\section{\label{sec:basic_idea}The basic idea of confluence relations}

In this section, we illustrate the basic mechanism of the confluence
relations, first in the classical setting and then in the $q$-setting.

\subsection{The classical case}

We first recall the classical construction of \cite{Hirose_Sato_confluence}
in terms of MPLs. For positive integers $k_{1},\dots,k_{d}$, write

\[
\operatorname{Li}_{k_{1},\dots,k_{d}}(x_{1},\dots,x_{d})\coloneqq\sum_{0<m_{1}<\cdots<m_{d}}\frac{x^{m_{1}}_{1}\cdots x^{m_{d}}_{d}}{m^{k_{1}}_{1}\cdots m^{k_{d}}_{d}}.
\]
The iterated integrals occurring in \cite{Hirose_Sato_confluence}
can be expressed as MPLs of this form, where
\[
x_{i}\in\{1,z,z^{-1}\},\qquad x_{i}\cdots x_{d}\in\{1,z\}\quad(1\leq i\leq d)
\]
and $(k_{d},x_{d})\neq(1,1)$. Their differential equations with respect
to $z$ express their derivatives in terms of MPLs for which the sum
$k_{1}+\cdots+k_{d}$ is smaller. By solving the resulting differential
equations recursively, one can express them as linear combinations
of MPLs of the form
\[
\operatorname{Li}_{k_{1},\dots,k_{d}}(1,\dots,1,z)
\]
with coefficients given by MZVs. In the terminology of \cite{Hirose_Sato_confluence},
the relations in the word algebra underlying these functional identities
are called standard relations. Their suitably regularized limits as
$z\to1$ give rise to the confluence relations among MZVs.

We illustrate this procedure by a simple example for which the final
limiting step requires no regularization. Consider 
\[
\mathrm{Li}_{2,2}(z,1).
\]
We first recursively differentiate the $z$-dependent MPLs that appear
until no $z$-dependent MPLs remain. The relevant differential equations
are 
\begin{align*}
\frac{d}{dz}\mathrm{Li}_{2,2}(z,1) & =\frac{\mathrm{Li}_{1,2}(z,1)}{z},\\
\frac{d}{dz}\mathrm{Li}_{1,2}(z,1) & =\frac{\zeta(2)}{1-z}-\frac{\mathrm{Li}_{2}(z)}{z(1-z)},\\
\frac{d}{dz}\mathrm{Li}_{2}(z) & =\frac{\mathrm{Li}_{1}(z)}{z},\\
\frac{d}{dz}\mathrm{Li}_{1}(z) & =\frac{1}{1-z}.
\end{align*}
We now integrate these equations back in reverse order. Since all
the MPLs occurring on the left-hand sides vanish at $z=0$, we obtain
\begin{align*}
\mathrm{Li}_{1}(z) & =\int^{z}_{0}\frac{dt}{1-t}=\mathrm{Li}_{1}(z),\\
\mathrm{Li}_{2}(z) & =\int^{z}_{0}\frac{\mathrm{Li}_{1}(t)}{t}\,dt=\mathrm{Li}_{2}(z),\\
\mathrm{Li}_{1,2}(z,1) & =\int^{z}_{0}\left(\frac{\zeta(2)}{1-t}-\frac{\mathrm{Li}_{2}(t)}{t(1-t)}\right)dt\\
 & =\zeta(2)\mathrm{Li}_{1}(z)-\mathrm{Li}_{3}(z)-\mathrm{Li}_{2,1}(1,z),\\
\mathrm{Li}_{2,2}(z,1) & =\int^{z}_{0}\frac{\mathrm{Li}_{1,2}(t,1)}{t}\,dt\\
 & =\zeta(2)\mathrm{Li}_{2}(z)-\mathrm{Li}_{4}(z)-\mathrm{Li}_{2,2}(1,z).
\end{align*}
Thus we have expressed the original MPL in terms of MPLs of the form
$\mathrm{Li}_{k_{1},\dots,k_{d}}(1,\dots,1,z)$ with MZV coefficients.
Letting $z\to1$, we obtain 
\[
\zeta(2,2)=\zeta(2)^{2}-\zeta(4)-\zeta(2,2).
\]

\subsection{The $q$-analogue}

We now describe the $q$-analogue of the preceding procedure. Among
the $q$MPLs 
\[
\mathrm{Li}_{q;k_{1},\dots,k_{d}}(x_{1},\dots,x_{d})
\]
reviewed in \Cref{sec:Review_qMPL}, we consider those for which each
argument $x_{i}$ is of the form $z^{h}q^{m}$ with $h\in\{0,\pm1\}$
and $m\in\mathbb{Z}$, and which satisfy 
\[
x_{i}\cdots x_{d}\in zq^{\mathbb{Z}}\cup q^{\mathbb{Z}_{>0}}\qquad(1\leq i\leq d).
\]
For the formulation of the confluence relations, we further restrict
this family to certain subclasses. We defer their precise definitions
to \Cref{sec:formulation_conf}. Their $q$-differential equations
with respect to $z$ express their $q$-derivatives in terms of $q$MPLs
for which the sum $k_{1}+\cdots+k_{d}$ is smaller. By solving the
resulting $q$-differential equations recursively, one can express
them as linear combinations of $q$MPLs of the form
\[
\mathrm{Li}_{q;k_{1},\dots,k_{d}}(q^{n_{1}},\dots,q^{n_{d-1}},zq^{n_{d}})
\]
with coefficients given by $z$-independent $q$MPL values. Taking
the limit as $z\to1$, with regularization when necessary, we obtain
relations among SZ $q$MZVs.

We illustrate this procedure by a simple example for which the final
limiting step requires no regularization. Consider 
\[
\mathrm{Li}_{q;2,2}(zq^{2},q^{2}).
\]
We first recursively apply $D_{q}$ to the $z$-dependent $q$MPLs
that appear until no $z$-dependent $q$MPLs remain. The relevant
$q$-differential equations are 
\begin{align*}
D_{q}\mathrm{Li}_{q;2,2}(zq^{2},q^{2}) & =\frac{\mathrm{Li}_{q;1,2}(zq^{2},q^{2})}{z},\\
D_{q}\mathrm{Li}_{q;1,2}(zq^{2},q^{2}) & =\frac{q^{2}}{1-zq^{2}}\mathrm{Li}_{q;2}(q^{2})-\frac{\mathrm{Li}_{q;2}(zq^{4})}{z(1-zq^{2})},\\
D_{q}\mathrm{Li}_{q;2}(zq^{4}) & =\frac{\mathrm{Li}_{q;1}(zq^{4})}{z},\\
D_{q}\mathrm{Li}_{q;1}(zq^{4}) & =\frac{q^{4}}{1-zq^{4}}.
\end{align*}
We now integrate these equations back in reverse order using the Jackson
integral. Since
\[
\int(D_{q}f)(z)\,d_{q}z=f(z)-f(0)
\]
and all the $q$MPLs occurring on the left-hand sides vanish at $z=0$,
we obtain 
\begin{align*}
\mathrm{Li}_{q;1}(zq^{4}) & =\int\frac{q^{4}}{1-zq^{4}}\,d_{q}z=\mathrm{Li}_{q;1}(zq^{4}),\\
\mathrm{Li}_{q;2}(zq^{4}) & =\int\frac{\mathrm{Li}_{q;1}(zq^{4})}{z}\,d_{q}z=\mathrm{Li}_{q;2}(zq^{4}),\\
\mathrm{Li}_{q;1,2}(zq^{2},q^{2}) & =\int\left(\frac{q^{2}}{1-zq^{2}}\mathrm{Li}_{q;2}(q^{2})-\frac{\mathrm{Li}_{q;2}(zq^{4})}{z(1-zq^{2})}\right)d_{q}z\\
 & =\mathrm{Li}_{q;2}(q^{2})\mathrm{Li}_{q;1}(zq^{2})-\mathrm{Li}_{q;3}(zq^{4})-\mathrm{Li}_{q;2,1}(q^{2},zq^{2}),\\
\mathrm{Li}_{q;2,2}(zq^{2},q^{2}) & =\int\frac{\mathrm{Li}_{q;1,2}(zq^{2},q^{2})}{z}\,d_{q}z\\
 & =\mathrm{Li}_{q;2}(q^{2})\mathrm{Li}_{q;2}(zq^{2})-\mathrm{Li}_{q;4}(zq^{4})-\mathrm{Li}_{q;2,2}(q^{2},zq^{2}).
\end{align*}
Thus we have expressed the original $q$MPL in terms of $q$MPLs of
the form 
\[
\mathrm{Li}_{q;k_{1},\dots,k_{d}}(q^{n_{1}},\dots,q^{n_{d-1}},zq^{n_{d}})
\]
with $z$-independent $q$MPL coefficients. Letting $z\to1$ and rewriting
the resulting special values in terms of SZ $q$MZVs (see \Cref{sec:qMZVs}
for the definition), we obtain 
\[
\zeta^{\mathrm{SZ}}_{q}(2,2)=\bigl(\zeta^{\mathrm{SZ}}_{q}(2)\bigr)^{2}-\zeta^{\mathrm{SZ}}_{q}(4)-\zeta^{\mathrm{SZ}}_{q}(2,2).
\]

\section{\label{sec:qMZVs}SZ $q$MZVs}

\subsection{Definition of SZ $q$MZVs\label{subsec:SZ-qMZVs}}

For $(k_{1},\dots,k_{d})\in\mathbb{Z}^{d}_{\geq0}$ with $d=0$ or
$k_{d}\geq1$, the SZ $q$MZVs are defined by 
\[
\zeta^{\mathrm{SZ}}_{q}(k_{1},\ldots,k_{d})\coloneqq\sum_{0<m_{1}<\cdots<m_{d}}\prod^{d}_{j=1}\left(\frac{q^{m_{j}}}{1-q^{m_{j}}}\right)^{k_{j}}\in\mathbb{C}.
\]
When $d=0$, we regard the right-hand side as $1$, so $\zeta^{\mathrm{SZ}}_{q}(\emptyset)=1$.

Let $\mathbb{Q}\langle\mathcal{B}\rangle$ be the free noncommutative
algebra generated by 
\[
\mathcal{B}\coloneqq\{b_{0},b_{1},b_{2},\ldots\}.
\]
We denote the empty word by $1$. Let $\mathbb{Q}\langle\mathcal{B}\rangle^{0}$
be the subspace spanned by $1$ and all words that do not end with
$b_{0}$. Define the $\mathbb{Q}$-linear map
\[
Z^{\mathrm{SZ}}_{q}\colon\mathbb{Q}\langle\mathcal{B}\rangle^{0}\to\mathbb{C}
\]
by
\[
Z^{\mathrm{SZ}}_{q}(b_{k_{1}}\cdots b_{k_{d}})\coloneqq\zeta^{\mathrm{SZ}}_{q}(k_{1},\ldots,k_{d}).
\]
The same series defines a formal map $Z^{\mathrm{SZ}}\colon\mathbb{Q}\langle\mathcal{B}\rangle^{0}\to\mathbb{Q}[[q]]$,
and its kernel is one of the main themes of the study of $q$MZVs.

\subsection{The stuffle product for SZ $q$MZVs}

The (SZ) stuffle product 
\[
*\colon\mathbb{Q}\langle\mathcal{B}\rangle\times\mathbb{Q}\langle\mathcal{B}\rangle\longrightarrow\mathbb{Q}\langle\mathcal{B}\rangle
\]
is the $\mathbb{Q}$-bilinear map defined by 
\begin{align*}
1*u & \coloneqq u*1\coloneqq u,\\
(b_{i}u)*(b_{j}v) & \coloneqq b_{i}(u*b_{j}v)+b_{j}(b_{i}u*v)+b_{i+j}(u*v),
\end{align*}
for $u,v\in\mathbb{Q}\langle\mathcal{B}\rangle$ and $i,j\geq0$.
By definition, $u*v\in\mathbb{Q}\langle\mathcal{B}\rangle^{0}$ for
$u,v\in\mathbb{Q}\langle\mathcal{B}\rangle^{0}$. Then, for all $u,v\in\mathbb{Q}\langle\mathcal{B}\rangle^{0}$,
we have 
\[
Z^{\mathrm{SZ}}_{q}(u*v)=Z^{\mathrm{SZ}}_{q}(u)Z^{\mathrm{SZ}}_{q}(v).
\]
Note that we have
\[
\bigl(\mathbb{Q}\langle\mathcal{B}\rangle,*\bigr)\simeq\bigl(\mathbb{Q}\langle\mathcal{B}\rangle^{0},*\bigr)[T]\quad;\quad b_{0}\mapsto T.
\]
Thus, every $u\in\mathbb{Q}\langle\mathcal{B}\rangle$ can be uniquely
expressed as
\[
u=u_{0}+u_{1}*b_{0}+\cdots+u_{k}*b^{*k}_{0}\qquad(u_{j}\in\mathbb{Q}\langle\mathcal{B}\rangle^{0}).
\]
We define the map
\[
\operatorname{reg}_{*}\colon\mathbb{Q}\langle\mathcal{B}\rangle\to\mathbb{Q}\langle\mathcal{B}\rangle^{0}
\]
by 
\[
\operatorname{reg}_{*}(u)\coloneqq u_{0}.
\]

\subsection{SZ $q$MPLs and the residue formula}

For $d\geq0$ and $k_{1},\dots,k_{d}\in\mathbb{Z}_{\geq0}$, define
the $q$MPLs of the SZ model by
\[
\mathrm{Li}^{\mathrm{SZ}}_{q;k_{1},\dots,k_{d}}(z)\coloneqq\sum_{0=m_{0}<m_{1}<\cdots<m_{d}}z^{m_{d}}\prod^{d}_{j=1}\left(\frac{q^{m_{j}}}{1-q^{m_{j}}}\right)^{k_{j}}.
\]
By definition, $\mathrm{Li}^{\mathrm{SZ}}_{q;\emptyset}(z)=1$. For
$d\geq1$, this defining series converges absolutely for $|z|<|q|^{-k_{d}}$.
Writing $\boldsymbol{k}=(k_{1},\dots,k_{d})\neq\emptyset$, these
functions satisfy the relations
\begin{align*}
\mathrm{Li}^{\mathrm{SZ}}_{q;\boldsymbol{k}_{\uparrow}}(z)-\mathrm{Li}^{\mathrm{SZ}}_{q;\boldsymbol{k}_{\uparrow}}(zq) & =\mathrm{Li}^{\mathrm{SZ}}_{q;\boldsymbol{k}}(zq),\\
\mathrm{Li}^{\mathrm{SZ}}_{q;\boldsymbol{k},0}(z) & =\frac{z}{1-z}\mathrm{Li}^{\mathrm{SZ}}_{q;\boldsymbol{k}}(z),
\end{align*}
where $\boldsymbol{k}_{\uparrow}\coloneqq(k_{1},\dots,k_{d-1},k_{d}+1)$.
These identities show that $\mathrm{Li}^{\mathrm{SZ}}_{q;k_{1},\dots,k_{d}}(z)$
admits a meromorphic continuation in $z$ to $\mathbb{C}$, with possible
poles at $z\in q^{\mathbb{Z}_{\leq0}}$. Define the $\mathbb{Q}$-linear
map
\[
L^{\mathrm{SZ}}_{q}\colon\mathbb{Q}\langle\mathcal{B}\rangle\to\mathcal{M}(\mathbb{C})
\]
by 
\[
L^{\mathrm{SZ}}_{q}(b_{k_{1}}\cdots b_{k_{d}})\coloneqq\mathrm{Li}^{\mathrm{SZ}}_{q;k_{1},\dots,k_{d}}(z).
\]
The following theorem may be viewed as a generalization of the identity
\[
\lim_{z\to1}L^{\mathrm{SZ}}_{q}(u)=Z^{\mathrm{SZ}}_{q}(u)
\]
for $u\in\mathbb{Q}\langle\mathcal{B}\rangle^{0}$, in which case
$L^{\mathrm{SZ}}_{q}(u)$ is holomorphic at $z=1$.
\begin{thm}
\label{thm:residue_and_stuffle_regularization}For $u\in\mathbb{Q}\langle\mathcal{B}\rangle$,
we have
\[
Z^{\mathrm{SZ}}_{q}\!\left(\operatorname{reg}_{*}(u)\right)=\operatorname*{Res}_{z=1}\frac{L^{\mathrm{SZ}}_{q}(u)(z)}{z(z-1)}\,dz.
\]
\end{thm}

See \Cref{subsec:Proof_of_stuffle_residue} for the proof of \Cref{thm:residue_and_stuffle_regularization}.

\subsection{Relation to general $q$MPLs}

Let $k_{1},\dots,k_{d},n_{1},\dots,n_{d}\in\mathbb{Z}_{\geq0}$ with
$n_{i}\leq k_{i}$, and consider the series
\[
\sum_{0=m_{0}<m_{1}<\cdots<m_{d}}z^{m_{d}}\prod^{d}_{i=1}\frac{q^{m_{i}n_{i}}}{(1-q^{m_{i}})^{k_{i}}}.
\]
When $d\geq1$ and $k_{1},\dots,k_{d}\geq1$, this series is the specialization
\[
\mathrm{Li}_{q;k_{1},\dots,k_{d}}(q^{n_{1}},\dots,q^{n_{d-1}},zq^{n_{d}})
\]
of the general $q$MPLs reviewed in the next section. We now express
it in terms of the SZ $q$MPLs. For each $i$, the binomial identity
\begin{align*}
\frac{q^{m_{i}n_{i}}}{(1-q^{m_{i}})^{k_{i}}} & =\left(1+\frac{1-q^{m_{i}}}{q^{m_{i}}}\right)^{k_{i}-n_{i}}\left(\frac{q^{m_{i}}}{1-q^{m_{i}}}\right)^{k_{i}}\\
 & =\sum^{k_{i}-n_{i}}_{s_{i}=0}\binom{k_{i}-n_{i}}{s_{i}}\left(\frac{q^{m_{i}}}{1-q^{m_{i}}}\right)^{k_{i}-s_{i}}
\end{align*}
expresses the $i$-th factor as a $\mathbb{Q}$-linear combination
of the factors occurring in the SZ model. Accordingly, define
\[
F(k_{1},\dots,k_{d};n_{1},\dots,n_{d})\coloneqq\sum^{k_{1}-n_{1}}_{s_{1}=0}\cdots\sum^{k_{d}-n_{d}}_{s_{d}=0}\left(\prod^{d}_{i=1}\binom{k_{i}-n_{i}}{s_{i}}\right)b_{k_{1}-s_{1}}\cdots b_{k_{d}-s_{d}}.
\]
Then
\[
\sum_{0<m_{1}<\cdots<m_{d}}z^{m_{d}}\prod^{d}_{i=1}\frac{q^{m_{i}n_{i}}}{(1-q^{m_{i}})^{k_{i}}}=L^{\mathrm{SZ}}_{q}(F(k_{1},\dots,k_{d};n_{1},\dots,n_{d})).
\]

\section{\label{sec:Review_qMPL}Review of $q$MPLs}

\subsection{Definition of $q$MPLs}

Following \cite{Zhao_qMPL}, we define $q$MPLs as follows. For $d\geq0$,
complex variables $x_{1},\dots,x_{d}$ and positive integers $k_{1},\dots,k_{d}$,
define the $q$MPL by
\[
\mathrm{Li}_{q;k_{1},\dots,k_{d}}(x_{1},\dots,x_{d})\coloneqq\sum_{0<m_{1}<\cdots<m_{d}}\frac{x^{m_{1}}_{1}\cdots x^{m_{d}}_{d}}{(1-q^{m_{1}})^{k_{1}}\cdots(1-q^{m_{d}})^{k_{d}}}.
\]
By definition, $\mathrm{Li}_{q;\emptyset}(\emptyset)=1$. This series
converges absolutely for $|x_{i}\cdots x_{d}|<1$ ($1\leq i\leq d$)
and admits a meromorphic continuation to $\mathbb{C}^{d}$ with possible
poles along
\[
x_{i}\cdots x_{d}\in q^{\mathbb{Z}_{\leq0}}\qquad(1\leq i\leq d).
\]

\subsection{$q$-difference formulas}

Let us recall the $q$-difference formulas for $q$MPLs from \cite{Zhao_qMPL}.
\begin{thm}[\cite{Zhao_qMPL}]
\label{thm:q-difference_of_qMPL}Let $j\in\{1,\dots,d\}$. If $k_{j}>1$,
then we have
\begin{align*}
 & \mathrm{Li}_{q;k_{1},\dots,k_{d}}(x_{1},\dots,x_{d})-\mathrm{Li}_{q;k_{1},\dots,k_{d}}(x_{1},\dots,qx_{j},\dots,x_{d})\\
 & =\mathrm{Li}_{q;k_{1},\dots,k_{j}-1,\dots,k_{d}}(x_{1},\dots,x_{d}).
\end{align*}
If $k_{j}=1$, then, assuming that $x_{j}\neq1$, we have
\begin{align*}
 & \mathrm{Li}_{q;k_{1},\dots,k_{d}}(x_{1},\dots,x_{d})-\mathrm{Li}_{q;k_{1},\dots,k_{d}}(x_{1},\dots,qx_{j},\dots,x_{d})\\
 & =\frac{x_{j}}{1-x_{j}}\mathrm{Li}_{q;k_{1},\dots,\widehat{k_{j}},\dots,k_{d}}(x_{1},\dots,x_{j-1}x_{j},\dots,x_{d})\\
 & \quad-\frac{1}{1-x_{j}}\mathrm{Li}_{q;k_{1},\dots,\widehat{k_{j}},\dots,k_{d}}(x_{1},\dots,x_{j}x_{j+1},\dots,x_{d}).
\end{align*}
Here, we regard the first term as $\frac{x_{1}}{1-x_{1}}\mathrm{Li}_{q;k_{2},\dots,k_{d}}(x_{2},\dots,x_{d})$
when $j=1$, and the second term as $0$ when $j=d$.
\end{thm}

\begin{cor}
\label{cor:Jackson_integral_qMPL}Let $d\geq1$. When $x_{1},\dots,x_{d}$
are constant with respect to $z$, we have
\begin{align*}
\int\frac{\mathrm{Li}_{q;k_{1},\dots,k_{d}}(x_{1},\dots,x_{d-1},x_{d}z)}{z}d_{q}z & =\mathrm{Li}_{q;k_{1},\dots,k_{d-1},k_{d}+1}(x_{1},\dots,x_{d-1},x_{d}z),\\
\int\frac{\mathrm{Li}_{q;k_{1},\dots,k_{d}}(x_{1},\dots,x_{d-1},x_{d}z)}{z-a^{-1}}d_{q}z & =-\mathrm{Li}_{q;k_{1},\dots,k_{d},1}(x_{1},\dots,x_{d-1},x_{d}a^{-1},az),\\
\int\frac{1}{z-a^{-1}}d_{q}z & =-\mathrm{Li}_{q;1}(az),
\end{align*}
where $a\in\mathbb{C}^{\times}$.
\end{cor}

\begin{example}
We give several examples illustrating \Cref{thm:q-difference_of_qMPL}.
\begin{align*}
 & \mathrm{Li}_{q;k_{1},k_{2},k_{3},k_{4},k_{5}}(x_{1},x_{2},x_{3},x_{4},x_{5})-\mathrm{Li}_{q;k_{1},k_{2},k_{3},k_{4},k_{5}}(x_{1},x_{2},qx_{3},x_{4},x_{5})\\
 & =\mathrm{Li}_{q;k_{1},k_{2},k_{3}-1,k_{4},k_{5}}(x_{1},x_{2},x_{3},x_{4},x_{5})\qquad(k_{3}>1).
\end{align*}
\begin{align*}
 & \mathrm{Li}_{q;k_{1},k_{2},1,k_{4},k_{5}}(x_{1},x_{2},x_{3},x_{4},x_{5})-\mathrm{Li}_{q;k_{1},k_{2},1,k_{4},k_{5}}(x_{1},x_{2},qx_{3},x_{4},x_{5})\\
 & =\frac{x_{3}}{1-x_{3}}\mathrm{Li}_{q;k_{1},k_{2},k_{4},k_{5}}(x_{1},x_{2}x_{3},x_{4},x_{5})-\frac{1}{1-x_{3}}\mathrm{Li}_{q;k_{1},k_{2},k_{4},k_{5}}(x_{1},x_{2},x_{3}x_{4},x_{5}).
\end{align*}
\begin{align*}
 & \mathrm{Li}_{q;1,k_{2},k_{3},k_{4},k_{5}}(x_{1},x_{2},x_{3},x_{4},x_{5})-\mathrm{Li}_{q;1,k_{2},k_{3},k_{4},k_{5}}(qx_{1},x_{2},x_{3},x_{4},x_{5})\\
 & =\frac{x_{1}}{1-x_{1}}\mathrm{Li}_{q;k_{2},k_{3},k_{4},k_{5}}(x_{2},x_{3},x_{4},x_{5})-\frac{1}{1-x_{1}}\mathrm{Li}_{q;k_{2},k_{3},k_{4},k_{5}}(x_{1}x_{2},x_{3},x_{4},x_{5}).
\end{align*}
\begin{align*}
 & \mathrm{Li}_{q;k_{1},k_{2},k_{3},k_{4},1}(x_{1},x_{2},x_{3},x_{4},x_{5})-\mathrm{Li}_{q;k_{1},k_{2},k_{3},k_{4},1}(x_{1},x_{2},x_{3},x_{4},qx_{5})\\
 & =\frac{x_{5}}{1-x_{5}}\mathrm{Li}_{q;k_{1},k_{2},k_{3},k_{4}}(x_{1},x_{2},x_{3},x_{4}x_{5}).
\end{align*}
\end{example}

\begin{example}
\label{exa:q_diff}In this paper, we consider one-variable specializations
of $q$MPLs, such as
\[
f(z)\coloneqq\mathrm{Li}_{q;3,5}(z^{-1},zq^{2}).
\]
The $q$-difference $f(z)-f(zq)$ can be computed by applying \Cref{thm:q-difference_of_qMPL}
successively to the two arguments. If we replace the first argument
before the second, then
\begin{align*}
f(z)-f(zq) & =\mathrm{Li}_{q;3,5}(z^{-1},zq^{2})-\mathrm{Li}_{q;3,5}(z^{-1}q^{-1},zq^{3})\\
 & =\left(\mathrm{Li}_{q;3,5}(z^{-1},zq^{2})-\mathrm{Li}_{q;3,5}(z^{-1}q^{-1},zq^{2})\right)\\
 & \quad+\left(\mathrm{Li}_{q;3,5}(z^{-1}q^{-1},zq^{2})-\mathrm{Li}_{q;3,5}(z^{-1}q^{-1},zq^{3})\right)\\
 & =-\mathrm{Li}_{q;2,5}(z^{-1}q^{-1},zq^{2})+\mathrm{Li}_{q;3,4}(z^{-1}q^{-1},zq^{2}).
\end{align*}
If we replace the second argument before the first, then
\begin{align*}
f(z)-f(zq) & =\left(\mathrm{Li}_{q;3,5}(z^{-1},zq^{2})-\mathrm{Li}_{q;3,5}(z^{-1},zq^{3})\right)\\
 & \quad+\left(\mathrm{Li}_{q;3,5}(z^{-1},zq^{3})-\mathrm{Li}_{q;3,5}(z^{-1}q^{-1},zq^{3})\right)\\
 & =\mathrm{Li}_{q;3,4}(z^{-1},zq^{2})-\mathrm{Li}_{q;2,5}(z^{-1}q^{-1},zq^{3}).
\end{align*}
\end{example}

\begin{example}
\label{exa:invalid_intermediate_term}When computing a $q$-difference
by repeated application of \Cref{thm:q-difference_of_qMPL}, one must
ensure that each intermediate $q$MPL is well defined, either by its
defining series or by meromorphic continuation. For example, the following
calculation is invalid:
\begin{align*}
\mathrm{Li}_{q;1,1}(z^{-1},zq)-\mathrm{Li}_{q;1,1}(z^{-1}q^{-1},zq^{2}) & =\left(\mathrm{Li}_{q;1,1}(z^{-1},zq)-\mathrm{Li}_{q;1,1}(z^{-1}q^{-1},zq)\right)\\
 & \quad+\left(\mathrm{Li}_{q;1,1}(z^{-1}q^{-1},zq)-\mathrm{Li}_{q;1,1}(z^{-1}q^{-1},zq^{2})\right)\\
 & =-\frac{z^{-1}q^{-1}}{1-z^{-1}q^{-1}}\mathrm{Li}_{q;1}(zq)+\frac{1}{1-z^{-1}q^{-1}}\mathrm{Li}_{q;1}(1)\\
 & \quad+\frac{zq}{1-zq}\mathrm{Li}_{q;1}(1)\\
 & =\frac{1}{1-zq}\mathrm{Li}_{q;1}(zq).
\end{align*}
Indeed, the intermediate term
\[
\mathrm{Li}_{q;1,1}(z^{-1}q^{-1},zq)
\]
lies on the polar locus $x_{1}x_{2}=1$, and $\mathrm{Li}_{q;1}(1)$
is not defined. In fact, this calculation leads to an incorrect result.
A valid computation is obtained by replacing the second argument first:
\begin{align*}
\mathrm{Li}_{q;1,1}(z^{-1},zq)-\mathrm{Li}_{q;1,1}(z^{-1}q^{-1},zq^{2}) & =\left(\mathrm{Li}_{q;1,1}(z^{-1},zq)-\mathrm{Li}_{q;1,1}(z^{-1},zq^{2})\right)\\
 & \quad+\left(\mathrm{Li}_{q;1,1}(z^{-1},zq^{2})-\mathrm{Li}_{q;1,1}(z^{-1}q^{-1},zq^{2})\right)\\
 & =\frac{zq}{1-zq}\mathrm{Li}_{q;1}(q)+\frac{1}{1-zq}\mathrm{Li}_{q;1}(zq^{2})-\frac{zq}{1-zq}\mathrm{Li}_{q;1}(q)\\
 & =\frac{1}{1-zq}\mathrm{Li}_{q;1}(zq^{2})\neq\frac{1}{1-zq}\mathrm{Li}_{q;1}(zq).
\end{align*}
\end{example}

\section{\label{sec:formulation_conf}Formulation of the confluence relations}

\subsection{Formal $q$MPLs and formal single-position $q$-difference}
\begin{defn}[Formal $q$MPLs]
\label{def:formal_qMPL}We denote by $\hat{\mathcal{F}}$ the set
of formal symbols
\[
\mathbb{L}_{k_{1},\dots,k_{d}}(x_{1},\dots,x_{d}),
\]
where $d\geq0$, $k_{1},\dots,k_{d}\in\mathbb{Z}_{\geq1}$, and
\[
x_{i}\in\{z^{h}q^{m}\mid h\in\{0,\pm1\},m\in\mathbb{Z}\}\qquad(1\leq i\leq d),
\]
satisfying
\begin{equation}
x_{i}\cdots x_{d}\in zq^{\mathbb{Z}}\cup q^{\mathbb{Z}_{>0}}\qquad(1\leq i\leq d).\label{eq:condition_qMPL}
\end{equation}
We write the case $d=0$ as $\mathbb{L}(\emptyset)$. We call the
elements of $\hat{\mathcal{F}}$ formal $q$MPLs. We call $k_{1}+\cdots+k_{d}$
the weight. We also call $k_{i}$ the weight of the $i$-th position.
In this section, we introduce several subsets of $\hat{\mathcal{F}}$.
Although each subset is defined when it first appears, their definitions
and inclusion relations are collected in \Cref{app:notation-summary}
for the reader's convenience. For $\mathcal{X}\subset\hat{\mathcal{F}}$,
we denote by $\mathbb{Q}\mathcal{X}$ the $\mathbb{Q}$-vector space
with basis $\mathcal{X}$.
\end{defn}

\begin{defn}[Evaluation map]
Define a $\mathbb{Q}$-linear map
\[
E\colon\mathbb{Q}\hat{\mathcal{F}}\longrightarrow\mathcal{M}(\mathbb{C})
\]
by
\[
E\!\left(\mathbb{L}_{k_{1},\dots,k_{d}}(x_{1},\dots,x_{d})\right)(z)\coloneqq\mathrm{Li}_{q;k_{1},\dots,k_{d}}\bigl(x_{1},\dots,x_{d}\bigr).
\]
\end{defn}

Since
\[
E\!\left(\mathbb{L}_{k_{1},\dots,k_{d}}(x_{1},\dots,x_{d})\right)(z)=\sum_{0=m_{0}<m_{1}<\cdots<m_{d}}\frac{\prod^{d}_{i=1}(x_{i}\cdots x_{d})^{m_{i}-m_{i-1}}}{(1-q^{m_{1}})^{k_{1}}\cdots(1-q^{m_{d}})^{k_{d}}},
\]
the condition \eqref{eq:condition_qMPL} ensures convergence when
$|z|$ is sufficiently small. Furthermore, 
\begin{equation}
E\!\left(\mathbb{L}_{k_{1},\dots,k_{d}}(x_{1},\dots,x_{d})\right)(0)=0\label{eq:eval_at_zero}
\end{equation}
unless all $x_{i}$ are $z$-independent. 

We say that $\mathbb{L}_{k_{1},\dots,k_{d}}(x_{1},\dots,x_{d})\in\hat{\mathcal{F}}$
and $\mathbb{L}_{k_{1},\dots,k_{d}}(x_{1}',\dots,x_{d}')\in\hat{\mathcal{F}}$
are adjacent if there exists $j\in\{1,\dots,d\}$ such that
\begin{itemize}
\item $x_{i}=x_{i}'$ for all $i\in\{1,\dots,d\}\setminus\{j\}$,
\item $x_{j}\in zq^{\mathbb{Z}}\cup z^{-1}q^{\mathbb{Z}}$,
\item $x_{j}/x_{j}'\in\{q^{\pm1}\}$.
\end{itemize}
Let
\[
\Omega\coloneqq\bigoplus_{b\in\{0\}\cup q^{\mathbb{Z}}}\mathbb{Q}e_{b}
\]
be the $\mathbb{Q}$-module generated by formal symbols $e_{b}$ with
$b\in\{0\}\cup q^{\mathbb{Z}}$. Motivated by \Cref{thm:q-difference_of_qMPL},
we make the following definition.
\begin{defn}[Formal single-position $q$-difference]
\label{def:D}For adjacent $u,v\in\hat{\mathcal{F}}$, define
\[
D(u,v)\in\mathbb{Q}\hat{\mathcal{F}}\otimes\Omega
\]
as follows. When
\[
u=\mathbb{L}_{k_{1},\dots,k_{d}}(x_{1},\dots,x_{d}),v=\mathbb{L}_{k_{1},\dots,k_{d}}(x_{1},\dots,qx_{j},\dots,x_{d}),
\]
we put
\[
D(u,v)\coloneqq\mathbb{L}_{k_{1},\dots,k_{j}-1,\dots,k_{d}}(x_{1},\dots,x_{d})\otimes e_{0}
\]
for $k_{j}>1$, and
\begin{align*}
D(u,v) & \coloneqq-\mathbb{L}_{k_{1},\dots,\widehat{k_{j}},\dots,k_{d}}(x_{1},\dots,x_{j-1}x_{j},\dots,x_{d})\otimes e_{q^{-m}}\\
 & \qquad+\mathbb{L}_{k_{1},\dots,\widehat{k_{j}},\dots,k_{d}}(x_{1},\dots,x_{j}x_{j+1},\dots,x_{d})\otimes(e_{q^{-m}}-e_{0})
\end{align*}
for $k_{j}=1$ and $x_{j}=q^{m}z$, and
\begin{align*}
D(u,v) & \coloneqq\mathbb{L}_{k_{1},\dots,\widehat{k_{j}},\dots,k_{d}}(x_{1},\dots,x_{j-1}x_{j},\dots,x_{d})\otimes(e_{q^{-m}}-e_{0})\\
 & \qquad-\mathbb{L}_{k_{1},\dots,\widehat{k_{j}},\dots,k_{d}}(x_{1},\dots,x_{j}x_{j+1},\dots,x_{d})\otimes e_{q^{-m}}
\end{align*}
for $k_{j}=1$ and $x_{j}=q^{-m}z^{-1}$. Here, we understand that
\[
\mathbb{L}_{k_{1},\dots,\widehat{k_{j}},\dots,k_{d}}(x_{1},\dots,x_{j-1}x_{j},\dots,x_{d})=\mathbb{L}_{k_{2},\dots,k_{d}}(x_{2},\dots,x_{d})
\]
when $j=1$, and
\[
\mathbb{L}_{k_{1},\dots,\widehat{k_{j}},\dots,k_{d}}(x_{1},\dots,x_{j}x_{j+1},\dots,x_{d})=0
\]
when $j=d$. When
\[
u=\mathbb{L}_{k_{1},\dots,k_{d}}(x_{1},\dots,x_{d}),v=\mathbb{L}_{k_{1},\dots,k_{d}}(x_{1},\dots,q^{-1}x_{j},\dots,x_{d}),
\]
we define $D(u,v)$ by
\[
D(u,v)\coloneqq-D(v,u).
\]
\end{defn}

\begin{defn}
Define the $\mathbb{Q}$-linear map
\[
E_{\Omega}\colon\mathbb{Q}\hat{\mathcal{F}}\otimes\Omega\longrightarrow\mathcal{M}(\mathbb{C})
\]
by
\[
E_{\Omega}(u\otimes e_{b})(z)\coloneqq\frac{z}{z-b}E(u)(z)
\]
for $u\in\mathbb{Q}\hat{\mathcal{F}}$ and $b\in\{0\}\cup q^{\mathbb{Z}}$.
\end{defn}

\begin{lem}
\label{lem:E_adjacent_diff}For adjacent $u,v\in\hat{\mathcal{F}}$,
we have
\[
E(u)-E(v)=E_{\Omega}\bigl(D(u,v)\bigr).
\]
\end{lem}

\begin{proof}
It is enough to consider the case
\[
u=\mathbb{L}_{k_{1},\dots,k_{d}}(x_{1},\dots,x_{d}),v=\mathbb{L}_{k_{1},\dots,k_{d}}(x_{1},\dots,qx_{j},\dots,x_{d}).
\]
When $k_{j}>1$, by \Cref{thm:q-difference_of_qMPL}, we have
\begin{align*}
E(u)-E(v) & =\mathrm{Li}_{q;k_{1},\dots,k_{j}-1,\dots,k_{d}}(x_{1},\dots,x_{d})\\
 & =E_{\Omega}(\mathbb{L}_{k_{1},\dots,k_{j}-1,\dots,k_{d}}(x_{1},\dots,x_{d})\otimes e_{0}).
\end{align*}
When $k_{j}=1$, by \Cref{thm:q-difference_of_qMPL}, we have
\begin{align*}
\frac{1}{z}\left(E(u)-E(v)\right) & =\frac{1}{z}\frac{x_{j}}{1-x_{j}}\mathrm{Li}_{q;k_{1},\dots,\widehat{k_{j}},\dots,k_{d}}(x_{1},\dots,x_{j-1}x_{j},\dots,x_{d})\\
 & \quad-\frac{1}{z}\frac{1}{1-x_{j}}\mathrm{Li}_{q;k_{1},\dots,\widehat{k_{j}},\dots,k_{d}}(x_{1},\dots,x_{j}x_{j+1},\dots,x_{d}).
\end{align*}
Furthermore, when $x_{j}=q^{m}z$, we have
\[
\frac{1}{z}\frac{x_{j}}{1-x_{j}}=-\frac{1}{z-q^{-m}},\quad-\frac{1}{z}\frac{1}{1-x_{j}}=\frac{1}{z-q^{-m}}-\frac{1}{z}
\]
and when $x_{j}=q^{-m}z^{-1}$, we have
\[
\frac{1}{z}\frac{x_{j}}{1-x_{j}}=\frac{1}{z-q^{-m}}-\frac{1}{z},\quad-\frac{1}{z}\frac{1}{1-x_{j}}=-\frac{1}{z-q^{-m}},
\]
which completes the proof.
\end{proof}

\subsection{Definition of $\mathcal{F}^{n}$ and standard paths}
\begin{defn}
For $n\in\mathbb{Z}$, we denote by $\mathcal{F}^{n}$ the subset
of $\hat{\mathcal{F}}$ consisting of the elements 
\[
\mathbb{L}_{k_{1},\dots,k_{d}}(x_{1},\dots,x_{d})\in\hat{\mathcal{F}}
\]
such that 
\[
x_{i}\in\{vq^{s}\,\mid\,v\in\{1,zq^{n},1/(zq^{n})\},\,0\leq s\leq k_{i}\}
\]
for $i=1,\dots,d$.
\end{defn}

We put 
\[
\mathcal{F}\coloneqq\bigcup_{n\in\mathbb{Z}}\mathcal{F}^{n}\subset\hat{\mathcal{F}}.
\]

Although elements of $\hat{\mathcal{F}}$ are formal symbols, we sometimes
write 
\[
u(z)=\mathbb{L}_{k_{1},\dots,k_{d}}(x_{1}(z),\dots,x_{d}(z))
\]
to emphasize the dependence of the arguments on $z$. In this notation,
we write 
\[
u(zq)\coloneqq\mathbb{L}_{k_{1},\dots,k_{d}}(x_{1}(zq),\dots,x_{d}(zq)).
\]

\begin{defn}[Standard path]
Let 
\[
u(z)\coloneqq\mathbb{L}_{k_{1},\dots,k_{d}}(x_{1}(z),\dots,x_{d}(z))\in\mathcal{F}.
\]
For each $z$-dependent position $i$, there is a unique integer $r$
such that
\[
x_{i}(z)=zq^{r}\qquad\text{or}\qquad x_{i}(z)=z^{-1}q^{1-r}.
\]
We call $r$ the level of position $i$. The levels are determined
by the original word $u(z)$ and remain fixed throughout the following
procedure. Starting from $u(z)$, we replace all $z$-dependent entries
one at a time, in increasing order of their levels. Entries having
the same level are replaced from right to left. At each step, the
replacement is 
\[
x_{i}(z)\quad\longmapsto\quad x_{i}(zq).
\]
The resulting finite sequence consists of formal $q$MPLs (\Cref{thm:Fn_StandardPath})
and is denoted by 
\[
\operatorname{Std}(u)=(u_{0},u_{1},\dots,u_{l}),
\]
where $u_{0}=u(z)$ and $u_{l}=u(zq)$. We call it the standard path
from $u(z)$ to $u(zq)$.
\end{defn}

\begin{example}
Let
\[
u(z)\coloneqq\mathbb{L}_{1,1}(z^{-1},zq)\in\mathcal{F}.
\]
Both positions have level $1$. Therefore, the second entry is replaced
first, and then the first entry is replaced. Hence
\[
\operatorname{Std}(u)=\bigl(\mathbb{L}_{1,1}(z^{-1},zq),\ \mathbb{L}_{1,1}(z^{-1},zq^{2}),\ \mathbb{L}_{1,1}(z^{-1}q^{-1},zq^{2})\bigr).
\]
\end{example}

The standard path has the following property, which will be useful
in defining the confluence relations.
\begin{thm}
\label{thm:Fn_StandardPath}Let $n\in\mathbb{Z}$ and $u\in\mathcal{F}^{n}$.
Let 
\[
\operatorname{Std}(u)=(u_{0},u_{1},\dots,u_{l})
\]
be the standard path from $u(z)$ to $u(zq)$. Then $u_{0},\dots,u_{l}$
are formal $q$MPLs. Furthermore, for $i=1,\dots,l$, the elements
$u_{i-1}$ and $u_{i}$ are adjacent and 
\[
D(u_{i-1},u_{i})\in\mathbb{Q}\mathcal{F}^{n}\otimes(\mathbb{Q}e_{0}\oplus\mathbb{Q}e_{q^{-n}})+\mathbb{Q}\mathcal{F}^{n+1}\otimes(\mathbb{Q}e_{0}\oplus\mathbb{Q}e_{q^{-(n+1)}}).
\]
\end{thm}

We give a proof of this theorem in \Cref{subsec:Proof_of_fn_standard}.

\subsection{Formal $q$-differential and Jackson integral operators}

\begin{defn}[Formal $q$-differential operator]
Define
\[
\partial\colon\mathcal{F}\to\mathbb{Q}\mathcal{F}\otimes\Omega
\]
by
\[
\partial(u)\coloneqq\sum^{l}_{i=1}D(u_{i-1},u_{i}),
\]
where 
\[
\operatorname{Std}(u)=(u_{0},u_{1},\dots,u_{l}).
\]
We extend $\partial$ uniquely to a $\mathbb{Q}$-linear map
\[
\partial\colon\mathbb{Q}\mathcal{F}\longrightarrow\mathbb{Q}\mathcal{F}\otimes\Omega.
\]
\end{defn}

Then, by the definitions of $\partial$ and \Cref{lem:E_adjacent_diff},
we have
\begin{equation}
E_{\Omega}(\partial u)(z)=E(u)(z)-E(u)(zq)\qquad(=zD_{q}E(u)).\label{eq:evaluation_partial}
\end{equation}

Define the sets of almost constant formal $q$MPLs by
\[
\hat{\mathcal{C}}_{z}\coloneqq\left\{ \mathbb{L}_{k_{1},\dots,k_{d}}(q^{n_{1}},\dots,q^{n_{d-1}},zq^{n_{d}})\,\middle|\,d\geq1,\,n_{i}\in\mathbb{Z}\ (1\leq i\leq d)\right\} \cup\{\mathbb{L}(\emptyset)\}
\]
and
\[
\mathcal{C}^{n}_{z}\coloneqq\left\{ \mathbb{L}_{k_{1},\dots,k_{d}}(q^{n_{1}},\dots,q^{n_{d-1}},zq^{n+n_{d}})\,\middle|\,d\geq1,\,0\leq n_{i}\leq k_{i}\ (1\leq i\leq d)\right\} \cup\{\mathbb{L}(\emptyset)\}
\]
for $n\in\mathbb{Z}$. Then $\mathcal{C}^{n}_{z}=\hat{\mathcal{C}}_{z}\cap\mathcal{F}^{n}$.
Motivated by \Cref{cor:Jackson_integral_qMPL}, we make the following
definition.

\begin{defn}[Formal Jackson integral operator]
Define the $\mathbb{Q}$-linear map $S\colon\mathbb{Q}\hat{\mathcal{C}}_{z}\otimes\Omega\to\mathbb{Q}\hat{\mathcal{C}}_{z}$
by 
\[
S(\mathbb{L}_{k_{1},\dots,k_{d}}(q^{n_{1}},\dots,q^{n_{d-1}},zq^{n_{d}})\otimes e_{0})\coloneqq\mathbb{L}_{k_{1},\dots,k_{d-1},k_{d}+1}(q^{n_{1}},\dots,q^{n_{d-1}},zq^{n_{d}}),
\]
\[
S(\mathbb{L}_{k_{1},\dots,k_{d}}(q^{n_{1}},\dots,q^{n_{d-1}},zq^{n_{d}})\otimes e_{q^{-r}})\coloneqq-\mathbb{L}_{k_{1},\dots,k_{d},1}(q^{n_{1}},\dots,q^{n_{d-1}},q^{n_{d}-r},zq^{r}),
\]
\[
S(\mathbb{L}(\emptyset)\otimes e_{q^{-r}})\coloneqq-\mathbb{L}_{1}(zq^{r}),
\]
\[
S(\mathbb{L}(\emptyset)\otimes e_{0})\coloneqq0.
\]
The last assignment is made by convention, and this exceptional value
is not used in the recursive construction of $\varphi$ introduced
later.
\end{defn}

By \Cref{cor:Jackson_integral_qMPL}, we have the following.
\begin{lem}
\label{lem:evaluation_S}Let 
\[
\xi\in\mathbb{Q}\hat{\mathcal{C}}_{z}\otimes\Omega
\]
and suppose that the coefficient of $\mathbb{L}(\emptyset)\otimes e_{0}$
in $\xi$ is zero. Then
\[
\int\frac{E_{\Omega}(\xi)(z)}{z}\,d_{q}z=E(S(\xi))(z).
\]
\end{lem}

\subsection{Definition of the confluence relations}

We denote by $\hat{\mathcal{C}}\subset\hat{\mathcal{F}}$ the subset
consisting of the elements which are independent of $z$. We also
put $\mathcal{C}\coloneqq\hat{\mathcal{C}}\cap\mathcal{F}$.
\begin{lem}
\label{lem:partial_no_c_e0}For $u\in\mathcal{F}$, we have
\[
\partial(u)\in\mathbb{Q}(\mathcal{F}\setminus\mathcal{C})\otimes\Omega+\mathbb{Q}\mathcal{C}\otimes\bigoplus_{b\in q^{\mathbb{Z}}}\mathbb{Q}e_{b}.
\]
\end{lem}

\begin{proof}
For adjacent $v,v'\in\hat{\mathcal{F}}$, by the definition of $D(v,v')$,
we have
\[
D(v,v')\in\mathbb{Q}(\hat{\mathcal{F}}\setminus\hat{\mathcal{C}})\otimes\Omega+\mathbb{Q}\hat{\mathcal{C}}\otimes\bigoplus_{b\in q^{\mathbb{Z}}}\mathbb{Q}e_{b}.
\]
The lemma is an immediate consequence of this.
\end{proof}

Define $\varphi\colon\mathbb{Q}\mathcal{F}\to\mathbb{Q}\mathcal{C}\otimes\mathbb{Q}\hat{\mathcal{C}}_{z}$
recursively by
\[
\varphi\coloneqq(\mathrm{id}_{\mathbb{Q}\mathcal{C}}\otimes S)\circ(\varphi\otimes\mathrm{id}_{\Omega})\circ\partial+\mathrm{Const},
\]
where $(\mathrm{id}_{\mathbb{Q}\mathcal{C}}\otimes S)\circ(\varphi\otimes\mathrm{id}_{\Omega})\circ\partial$
is the composite map
\[
\mathbb{Q}\mathcal{F}\xrightarrow{\partial}\mathbb{Q}\mathcal{F}\otimes\Omega\xrightarrow{\varphi\otimes\mathrm{id}_{\Omega}}\mathbb{Q}\mathcal{C}\otimes\mathbb{Q}\hat{\mathcal{C}}_{z}\otimes\Omega\xrightarrow{\mathrm{id}_{\mathbb{Q}\mathcal{C}}\otimes S}\mathbb{Q}\mathcal{C}\otimes\mathbb{Q}\hat{\mathcal{C}}_{z}
\]
and $\mathrm{Const}$ is the $\mathbb{Q}$-linear map defined by
\[
\mathrm{Const}(u)\coloneqq\begin{cases}
u\otimes\mathbb{L}(\emptyset) & \text{if }u\in\mathcal{C}\\
0 & \text{if }u\in\mathcal{F}\setminus\mathcal{C}
\end{cases}
\]
for $u\in\mathcal{F}$. This recursion terminates because $\partial$
lowers the weight by one. Here, by \Cref{lem:partial_no_c_e0}, the
term $S(\mathbb{L}(\emptyset)\otimes e_{0})$ does not occur in the
recursive computation of $\varphi$.

By \Cref{thm:Fn_StandardPath} and the definition of $S$, we have
\[
\partial(\mathcal{F}^{n})\subset\mathbb{Q}\mathcal{F}^{n}\otimes(\mathbb{Q}e_{0}\oplus\mathbb{Q}e_{q^{-n}})+\mathbb{Q}\mathcal{F}^{n+1}\otimes(\mathbb{Q}e_{0}\oplus\mathbb{Q}e_{q^{-(n+1)}})
\]
and
\[
S(\mathbb{Q}\mathcal{C}^{n}_{z}\otimes(\mathbb{Q}e_{0}\oplus\mathbb{Q}e_{q^{-n}})+\mathbb{Q}\mathcal{C}^{n+1}_{z}\otimes(\mathbb{Q}e_{0}\oplus\mathbb{Q}e_{q^{-(n+1)}}))\subset\mathbb{Q}\mathcal{C}^{n}_{z}
\]
for $n\in\mathbb{Z}$. These inclusions imply
\begin{equation}
\varphi(\mathbb{Q}\mathcal{F}^{n})\subset\mathbb{Q}\mathcal{C}\otimes\mathbb{Q}\mathcal{C}^{n}_{z}\label{eq:phi_Fn_image}
\end{equation}
for $n\in\mathbb{Z}$.

We denote by $E^{(2)}$ the $\mathbb{Q}$-linear map
\[
\mathbb{Q}\hat{\mathcal{F}}\otimes\mathbb{Q}\hat{\mathcal{F}}\longrightarrow\mathcal{M}(\mathbb{C})\qquad;\qquad u\otimes v\mapsto E(u)E(v).
\]
 
\begin{prop}
\label{prop:evaluation_phi}For $u\in\mathbb{Q}\mathcal{F}$, we have
$E^{(2)}(\varphi(u))=E(u)$.
\end{prop}

\begin{proof}
By $\mathbb{Q}$-linearity, it is enough to consider the case $u\in\mathcal{F}$.
We prove the claim by induction on the weight of $u$. Write
\[
\partial u=\sum_{i}u_{i}\otimes\omega_{i}\qquad(u_{i}\in\mathcal{F},\,\omega_{i}\in\Omega)
\]
and
\[
\varphi(u_{i})=\sum_{j}u_{i,j}\otimes v_{i,j}\qquad(u_{i,j}\in\mathbb{Q}\mathcal{C},\,v_{i,j}\in\mathbb{Q}\hat{\mathcal{C}}_{z}).
\]
By the induction hypothesis, we have
\begin{equation}
\sum_{j}E(u_{i,j})E(v_{i,j})=E(u_{i}).\label{eq:E_ui}
\end{equation}
Then,
\begin{align*}
E^{(2)}(\varphi(u))(z) & =\sum_{i,j}E(u_{i,j})E(S(v_{i,j}\otimes\omega_{i}))+E^{(2)}(\mathrm{Const}(u))(z)\qquad(\text{by the recursive definition of \ensuremath{\varphi}})\\
 & =\sum_{i,j}\int\frac{E(u_{i,j})E_{\Omega}(v_{i,j}\otimes\omega_{i})}{z}d_{q}z+E^{(2)}(\mathrm{Const}(u))(z)\qquad(\text{by \cref{lem:evaluation_S} and \cref{lem:partial_no_c_e0}})\\
 & =\int\frac{E_{\Omega}(\partial u)}{z}d_{q}z+E^{(2)}(\mathrm{Const}(u))(z)\qquad(\text{by \eqref{eq:E_ui}}).
\end{align*}
By \eqref{eq:evaluation_partial}, we have
\[
\int\frac{E_{\Omega}(\partial u)}{z}\,d_{q}z=\int D_{q}E(u)\,d_{q}z=E(u)(z)-E(u)(0).
\]
Moreover,
\[
E^{(2)}(\mathrm{Const}(u))(z)=E(u)(0)
\]
by \eqref{eq:eval_at_zero} and the definition of $\mathrm{Const}$.
Combining these identities gives $E^{(2)}(\varphi(u))=E(u)$.
\end{proof}

By \Cref{prop:evaluation_phi}, for $u\in\mathbb{Q}\mathcal{F}$,
we have
\begin{equation}
\operatorname*{Res}_{z=1}\frac{E^{(2)}(\varphi(u))(z)}{z(z-1)}\,dz=\operatorname*{Res}_{z=1}\frac{E(u)(z)}{z(z-1)}\,dz.\label{eq:pre_conf}
\end{equation}
Let us express both sides of \eqref{eq:pre_conf} in terms of SZ $q$MZVs
for some $u$.

\begin{defn}
Define $\pi\colon\mathbb{Q}\mathcal{C}\to\mathbb{Q}\langle\mathcal{B}\rangle^{0}$
and $\pi_{z}\colon\mathbb{Q}\mathcal{C}^{0}_{z}\to\mathbb{Q}\langle\mathcal{B}\rangle$
by
\begin{align*}
\pi(\mathbb{L}_{k_{1},\dots,k_{d}}(q^{n_{1}},\dots,q^{n_{d}})) & \coloneqq\pi_{z}(\mathbb{L}_{k_{1},\dots,k_{d}}(q^{n_{1}},\dots,q^{n_{d-1}},q^{n_{d}}z))\\
 & \coloneqq F(k_{1},\dots,k_{d};n_{1},\dots,n_{d})
\end{align*}
and
\[
\pi(\mathbb{L}(\emptyset))\coloneqq\pi_{z}(\mathbb{L}(\emptyset))\coloneqq1.
\]
\end{defn}

\begin{defn}
Define the map
\[
\varphi^{\pi}\colon\mathbb{Q}\mathcal{F}^{0}\longrightarrow\mathbb{Q}\langle\mathcal{B}\rangle^{0}\otimes\mathbb{Q}\langle\mathcal{B}\rangle
\]
by
\[
\varphi^{\pi}\coloneqq(\pi\otimes\pi_{z})\circ\varphi.
\]
This is well-defined by \eqref{eq:phi_Fn_image}.
\end{defn}

We denote by $\left(Z^{\mathrm{SZ}}_{q}\right)^{(2)}$ the $\mathbb{Q}$-linear
map
\[
\mathbb{Q}\langle\mathcal{B}\rangle^{0}\otimes\mathbb{Q}\langle\mathcal{B}\rangle^{0}\longrightarrow\mathbb{C}\qquad;\qquad u\otimes v\mapsto Z^{\mathrm{SZ}}_{q}(u)Z^{\mathrm{SZ}}_{q}(v).
\]

\begin{lem}
\label{lem:rescal_1}For $u\in\mathbb{Q}\mathcal{F}^{0}$, we have
\[
\operatorname*{Res}_{z=1}\frac{E^{(2)}(\varphi(u))(z)}{z(z-1)}\,dz=\left(Z^{\mathrm{SZ}}_{q}\right)^{(2)}\circ(\mathrm{id}\otimes\operatorname{reg}_{*})\circ\varphi^{\pi}(u).
\]
\end{lem}

\begin{proof}
By the definitions of $\pi$, $\pi_{z}$, and $F$, we have
\[
E(a)=Z^{\mathrm{SZ}}_{q}(\pi(a)),\qquad E(v)=L^{\mathrm{SZ}}_{q}(\pi_{z}(v))
\]
for $a\in\mathbb{Q}\mathcal{C}$ and $v\in\mathbb{Q}\mathcal{C}^{0}_{z}$,
respectively. Write
\[
\varphi(u)=\sum_{j}a_{j}\otimes v_{j},\qquad a_{j}\in\mathbb{Q}\mathcal{C},\quad v_{j}\in\mathbb{Q}\mathcal{C}^{0}_{z}.
\]
Then \Cref{thm:residue_and_stuffle_regularization} gives
\begin{align*}
\operatorname*{Res}_{z=1}\frac{E^{(2)}(\varphi(u))(z)}{z(z-1)}\,dz & =\sum_{j}Z^{\mathrm{SZ}}_{q}(\pi(a_{j}))\operatorname*{Res}_{z=1}\frac{L^{\mathrm{SZ}}_{q}(\pi_{z}(v_{j}))(z)}{z(z-1)}\,dz\\
 & =\sum_{j}Z^{\mathrm{SZ}}_{q}(\pi(a_{j}))Z^{\mathrm{SZ}}_{q}\!\left(\operatorname{reg}_{*}(\pi_{z}(v_{j}))\right)\\
 & =\left(Z^{\mathrm{SZ}}_{q}\right)^{(2)}\circ(\mathrm{id}\otimes\operatorname{reg}_{*})\circ\varphi^{\pi}(u).\qedhere
\end{align*}
\end{proof}

We put
\[
\mathcal{G}\coloneqq\{\mathbb{L}_{k_{1},\dots,k_{d}}(x_{1},\dots,x_{d})\in\mathcal{F}^{0}\,\mid\,d=0\text{ or }x_{d}\neq z\}.
\]
For $u\in\mathbb{Q}\mathcal{G}$, we write $\left.u\right|_{z=1}\in\mathbb{Q}\mathcal{C}$
for the element obtained by setting $z=1$ in $u$.

\begin{lem}
\label{lem:rescal_2}For $u\in\mathbb{Q}\mathcal{G}$, we have
\[
\operatorname*{Res}_{z=1}\frac{E(u)(z)}{z(z-1)}\,dz=Z^{\mathrm{SZ}}_{q}(\pi(\left.u\right|_{z=1})).
\]
\end{lem}

\begin{proof}
Since $u\in\mathbb{Q}\mathcal{G}$, the function $E(u)(z)$ is holomorphic
at $z=1$, and $E(u)(1)=E(\left.u\right|_{z=1})$. Hence, by the definition
of $\pi$, 
\[
\operatorname*{Res}_{z=1}\frac{E(u)(z)}{z(z-1)}\,dz=E(u)(1)=E(\left.u\right|_{z=1})=Z^{\mathrm{SZ}}_{q}\!\left(\pi\!\left(\left.u\right|_{z=1}\right)\right).\qedhere
\]
\end{proof}

We now define the confluence relations. In the following definition,
we denote by $1_{\mathbb{Q}\langle\mathcal{B}\rangle^{0}}$ the unit
of $\mathbb{Q}\langle\mathcal{B}\rangle^{0}$, and $1_{\mathcal{S}}$
the unit of the symmetric algebra of $\mathbb{Q}\langle\mathcal{B}\rangle^{0}$.

\begin{defn}[Confluence relations as algebraic relations]
\label{def:conf_alg}Define $I^{\mathrm{alg}}_{\mathrm{CF}}$ as
the ideal of the symmetric algebra $\mathcal{S}(\mathbb{Q}\langle\mathcal{B}\rangle^{0})$
over $\mathbb{Q}$ generated by $\{1_{\mathbb{Q}\langle\mathcal{B}\rangle^{0}}-1_{\mathcal{S}}\}$
and 
\[
\{m_{\odot}\circ(\mathrm{id}\otimes\operatorname{reg}_{*})\circ\varphi^{\pi}(u)-\pi(\left.u\right|_{z=1})\,\mid\,u\in\mathbb{Q}\mathcal{G}\},
\]
where
\[
m_{\odot}\colon\mathbb{Q}\langle\mathcal{B}\rangle^{0}\otimes\mathbb{Q}\langle\mathcal{B}\rangle^{0}\to\mathcal{S}(\mathbb{Q}\langle\mathcal{B}\rangle^{0})
\]
is the $\mathbb{Q}$-linear map defined by $m_{\odot}(u\otimes v)\coloneqq u\odot v$
where $\odot$ denotes the product in the symmetric algebra.
\end{defn}

\begin{defn}[Confluence relations as linear relations]
We define the $\mathbb{Q}$-vector subspace $I_{\mathrm{CF}}\subset\mathbb{Q}\langle\mathcal{B}\rangle^{0}$
by
\[
I_{\mathrm{CF}}\coloneqq I^{\mathrm{alg}}_{\mathrm{CF}}\cap\mathbb{Q}\langle\mathcal{B}\rangle^{0}.
\]
\end{defn}

Let $\widetilde{Z}^{\mathrm{SZ}}_{q}\colon\mathcal{S}(\mathbb{Q}\langle\mathcal{B}\rangle^{0})\longrightarrow\mathbb{C}$
(resp. $\widetilde{Z}^{\mathrm{SZ}}\colon\mathcal{S}(\mathbb{Q}\langle\mathcal{B}\rangle^{0})\longrightarrow\mathbb{Q}[[q]]$)
be the unique $\mathbb{Q}$-algebra homomorphism extending $Z^{\mathrm{SZ}}_{q}$
(resp. $Z^{\mathrm{SZ}}$). 
\begin{thm}
We have 
\[
I^{\mathrm{alg}}_{\mathrm{CF}}\subset\ker\widetilde{Z}^{\mathrm{SZ}}
\]
 and 
\[
I_{\mathrm{CF}}\subset\ker Z^{\mathrm{SZ}}.
\]
\end{thm}

\begin{proof}
By \eqref{eq:pre_conf}, \Cref{lem:rescal_2} and \Cref{lem:rescal_1},
we have
\[
\widetilde{Z}^{\mathrm{SZ}}_{q}(u)=0
\]
for any $u\in I^{\mathrm{alg}}_{\mathrm{CF}}$. Since $q$ is an arbitrary
complex number satisfying $0<|q|<1$, we also have
\[
\widetilde{Z}^{\mathrm{SZ}}(u)=0.
\]
Hence, we have the first inclusion $I^{\mathrm{alg}}_{\mathrm{CF}}\subset\ker\widetilde{Z}^{\mathrm{SZ}}$.
The second inclusion is immediate by definition.
\end{proof}

\begin{conjecture}
We have 
\[
I^{\mathrm{alg}}_{\mathrm{CF}}=\ker\widetilde{Z}^{\mathrm{SZ}}
\]
 and 
\[
I_{\mathrm{CF}}=\ker Z^{\mathrm{SZ}}.
\]
\end{conjecture}

\begin{rem}
\label{rem:conf_alg_lin}Although $I_{\mathrm{CF}}$ is defined as
the linear part of $I^{\mathrm{alg}}_{\mathrm{CF}}$, the full algebraic
ideal can be recovered from it after identifying the symmetric product
with the stuffle product. More precisely, let 
\[
m_{*}\colon\mathcal{S}(\mathbb{Q}\langle\mathcal{B}\rangle^{0})\to(\mathbb{Q}\langle\mathcal{B}\rangle^{0},*)
\]
 be the unique $\mathbb{Q}$-algebra homomorphism whose restriction
to $\mathbb{Q}\langle\mathcal{B}\rangle^{0}$ is the identity. We
will prove in \Cref{cor:alg_lin_conf_by_stuffle} below that
\[
I^{\mathrm{alg}}_{\mathrm{CF}}=m^{-1}_{*}(I_{\mathrm{CF}}).
\]
\end{rem}

\section{\label{sec:Conf_and_Stuffle}Confluence relations and the stuffle
product}

Define the stuffle product
\[
*\colon\mathbb{Q}\mathcal{C}\times\mathbb{Q}\hat{\mathcal{F}}\to\mathbb{Q}\hat{\mathcal{F}}
\]
as the $\mathbb{Q}$-bilinear map determined by
\[
\mathbb{L}_{k_{1},\dots,k_{r}}(x_{1},\dots,x_{r})*\mathbb{L}_{l_{1},\dots,l_{s}}(y_{1},\dots,y_{s})=\sum^{r+s}_{t=\max(r,s)}\sum_{(f,g)\in\mathrm{Stu}_{t}(r,s)}\mathbb{L}_{m_{1},\dots,m_{t}}(z_{1},\dots,z_{t}),
\]
where $\mathrm{Stu}_{t}(r,s)$ is the set of pairs of strictly increasing
functions $f\colon\{1,\dots,r\}\to\{1,\dots,t\}$ and $g\colon\{1,\dots,s\}\to\{1,\dots,t\}$
such that
\[
\mathrm{Im}(f)\cup\mathrm{Im}(g)=\{1,\dots,t\},
\]
and
\[
(m_{i},z_{i})\coloneqq\begin{cases}
(k_{j},x_{j}) & \text{if }f^{-1}(\{i\})=\{j\}\text{ and }g^{-1}(\{i\})=\emptyset\\
(l_{j},y_{j}) & \text{if }f^{-1}(\{i\})=\emptyset\text{ and }g^{-1}(\{i\})=\{j\}\\
(k_{j}+l_{j'},x_{j}y_{j'}) & \text{if }f^{-1}(\{i\})=\{j\}\text{ and }g^{-1}(\{i\})=\{j'\}.
\end{cases}
\]
For example, 
\[
\mathbb{L}_{k}(x)*\mathbb{L}_{l}(y)=\mathbb{L}_{k,l}(x,y)+\mathbb{L}_{l,k}(y,x)+\mathbb{L}_{k+l}(xy).
\]
By definition, we also have
\[
u*\mathbb{L}(\emptyset)=u,\qquad\mathbb{L}(\emptyset)*v=v
\]
for $u\in\mathbb{Q}\mathcal{C}$ and $v\in\mathbb{Q}\mathcal{F}$.
This stuffle product is compatible with the stuffle product on $\mathbb{Q}\langle\mathcal{B}\rangle$;
that is, we have
\begin{equation}
\pi(u*v)=\pi(u)*\pi(v)\label{eq:pi_and_stuffle}
\end{equation}
for $u,v\in\mathbb{Q}\mathcal{C}$. Furthermore, for $u\in\mathbb{Q}\mathcal{C}$,
we write $l^{*}_{u}$ for the map $\mathbb{Q}\mathcal{F}\to\mathbb{Q}\mathcal{F}$
defined by $v\mapsto u*v$.
\begin{lem}
\label{lem:der_and_stuffle}For $u\in\mathbb{Q}\mathcal{C}$, we have
the following identity of maps from $\mathbb{Q}(\mathcal{F}\cap\hat{\mathcal{C}}_{z})$
to $\mathbb{Q}\mathcal{F}\otimes\Omega$:
\[
\partial\circ l^{*}_{u}=(l^{*}_{u}\otimes\mathrm{id}_{\Omega})\circ\partial.
\]
\end{lem}

A proof is given in \Cref{subsec:proof_der_and_stuffle}.

\begin{lem}
\label{lem:phi_and_stuffle}For $u\in\mathbb{Q}\mathcal{C}$ and $v\in\mathbb{Q}(\mathcal{F}\cap\hat{\mathcal{C}}_{z})$,
we have
\[
\varphi(u*v)=u\otimes v.
\]
\end{lem}

\begin{proof}
By $\mathbb{Q}$-bilinearity, it suffices to consider the case $v\in\mathcal{F}\cap\hat{\mathcal{C}}_{z}$.
We proceed by induction on the weight of $v$. The case $v=\mathbb{L}(\emptyset)$
is obvious since $\varphi(u*\mathbb{L}(\emptyset))=\varphi(u)=u\otimes\mathbb{L}(\emptyset)$.
Suppose that $v\neq\mathbb{L}(\emptyset)$. Then, every term of $u*v$
depends on $z$, and hence $\operatorname{Const}(u*v)=0$. Furthermore,
$\partial v\in\mathbb{Q}(\mathcal{F}\cap\hat{\mathcal{C}}_{z})\otimes\Omega$.
Therefore, we have
\[
\begin{aligned} & \varphi(u*v)\\
 & =(\operatorname{id}_{\mathbb{Q}\mathcal{C}}\otimes S)\circ(\varphi\otimes\operatorname{id}_{\Omega})\circ\partial(u*v)\qquad\qquad(\text{by \ensuremath{\operatorname{Const}(u*v)=0} and the definition of \ensuremath{\varphi}})\\
 & =(\operatorname{id}_{\mathbb{Q}\mathcal{C}}\otimes S)\circ(\varphi\otimes\operatorname{id}_{\Omega})\circ(l^{*}_{u}\otimes\operatorname{id}_{\Omega})\circ\partial(v)\qquad\qquad(\text{by \cref{lem:der_and_stuffle}})\\
 & =u\otimes S(\partial(v))\qquad\qquad(\text{by the induction hypothesis})\\
 & =u\otimes v\qquad\qquad(\text{by definition of \ensuremath{\partial} and \ensuremath{S}}).
\end{aligned}
\]
This completes the induction.
\end{proof}

\begin{thm}
\label{thm:stuffle_in_conf}We have
\[
u\odot v-u*v\in I^{\mathrm{alg}}_{\mathrm{CF}}
\]
for $u,v\in\mathbb{Q}\langle\mathcal{B}\rangle^{0}$, where $\odot$
denotes the product in the symmetric algebra. 
\end{thm}

\begin{proof}
Since $\pi(\mathbb{Q}\mathcal{C})=\mathbb{Q}\langle\mathcal{B}\rangle^{0}$
and $\pi_{z}(\mathbb{Q}(\mathcal{C}^{0}_{z}\cap\mathcal{G}))=\mathbb{Q}\langle\mathcal{B}\rangle^{0}$,
there exist $\hat{u}\in\mathbb{Q}\mathcal{C}$ and $\hat{v}\in\mathbb{Q}(\mathcal{C}^{0}_{z}\cap\mathcal{G})$
such that $\pi(\hat{u})=u$ and $\pi_{z}(\hat{v})=v$. Then, 
\[
m_{\odot}\circ(\mathrm{id}\otimes\operatorname{reg}_{*})\circ\varphi^{\pi}(\hat{u}*\hat{v})-\pi(\left.\hat{u}*\hat{v}\right|_{z=1})\in I^{\mathrm{alg}}_{\mathrm{CF}}.
\]
By \Cref{lem:phi_and_stuffle}, we have
\[
m_{\odot}\circ(\mathrm{id}\otimes\operatorname{reg}_{*})\circ\varphi^{\pi}(\hat{u}*\hat{v})=u\odot v.
\]
Furthermore, by \eqref{eq:pi_and_stuffle}, we have
\[
\pi(\left.\hat{u}*\hat{v}\right|_{z=1})=u*v.
\]
Hence, the theorem is proved.
\end{proof}

\begin{cor}
\label{cor:alg_lin_conf_by_stuffle}Let $m_{*}$ be as in \Cref{rem:conf_alg_lin}.
Then $m_{*}$ induces a $\mathbb{Q}$-linear isomorphism
\[
I^{\mathrm{alg}}_{\mathrm{CF}}/\ker(m_{*})\simeq I_{\mathrm{CF}}.
\]
Equivalently,
\[
m_{*}(I^{\mathrm{alg}}_{\mathrm{CF}})=I_{\mathrm{CF}},\qquad I^{\mathrm{alg}}_{\mathrm{CF}}=m^{-1}_{*}(I_{\mathrm{CF}}).
\]
In particular, $I_{\mathrm{CF}}$ is an ideal of $(\mathbb{Q}\langle\mathcal{B}\rangle^{0},*)$.
\end{cor}

Indeed, $\ker(m_{*})$ is generated by
\[
1_{\mathbb{Q}\langle\mathcal{B}\rangle^{0}}-1_{\mathcal{S}}\qquad\text{and}\qquad u\odot v-u*v\quad(u,v\in\mathbb{Q}\langle\mathcal{B}\rangle^{0}),
\]
so the assertions follow from \Cref{def:conf_alg} and \Cref{thm:stuffle_in_conf}.

\section{\label{sec:Conf_and_Duality}A conjectural approach to confluence
relations and dualities}

In this section, we show, assuming a certain conjectural identity,
that the duality relations for $q$MZVs in the BZ and SZ models are
contained in the confluence relations. In \Cref{subsec:BZ_SZ_duality},
after recalling the duality relations for $q$MZVs in the BZ and SZ
models, we define the corresponding $\mathbb{Q}$-vector subspaces
of $\mathbb{Q}\langle\mathcal{B}\rangle^{0}$. In \Cref{subsec:conjectural_duality},
we formulate the conjectural identity. In \Cref{subsec:conf_and_duality},
we derive the stated inclusion.

\subsection{\label{subsec:BZ_SZ_duality}BZ and SZ dualities}

For $a\in\{0,1\}$ and $r\geq0$, let $\{a\}^{r}$ denote $r$ consecutive
copies of $a$. For $\boldsymbol{k}=(k_{1},\ldots,k_{d})\in\mathbb{Z}^{d}_{\geq0}$
with $d\geq0$, put
\[
b_{\boldsymbol{k}}\coloneqq b_{k_{1}}\cdots b_{k_{d}}.
\]

A BZ-admissible index is a tuple $\boldsymbol{k}=(k_{1},\ldots,k_{d})\in\mathbb{Z}^{d}_{\geq1}$
such that $d=0$ or $k_{d}\geq2$. We denote the set of BZ-admissible
indices by $\operatorname{Ind}_{\mathrm{BZ}}$. Every BZ-admissible
index can be written uniquely as
\[
\boldsymbol{k}=\bigl(\{1\}^{d_{1}-1},c_{1}+1,\ldots,\{1\}^{d_{r}-1},c_{r}+1\bigr),\qquad c_{i},d_{i}\geq1.
\]
Define its BZ dual by
\[
\boldsymbol{k}^{\dagger}\coloneqq\bigl(\{1\}^{c_{r}-1},d_{r}+1,\ldots,\{1\}^{c_{1}-1},d_{1}+1\bigr)\in\operatorname{Ind}_{\mathrm{BZ}}.
\]

An SZ-admissible index is a tuple $\boldsymbol{k}=(k_{1},\ldots,k_{d})\in\mathbb{Z}^{d}_{\geq0}$
such that $d=0$ or $k_{d}\geq1$. We denote the set of SZ-admissible
indices by $\operatorname{Ind}_{\mathrm{SZ}}$. Every SZ-admissible
index can be written uniquely as
\[
\boldsymbol{k}=\bigl(\{0\}^{d_{1}-1},c_{1},\ldots,\{0\}^{d_{r}-1},c_{r}\bigr),\qquad c_{i},d_{i}\geq1.
\]
Define its SZ dual by
\[
\boldsymbol{k}^{\vee}\coloneqq\bigl(\{0\}^{c_{r}-1},d_{r},\ldots,\{0\}^{c_{1}-1},d_{1}\bigr)\in\operatorname{Ind}_{\mathrm{SZ}}.
\]

For $\boldsymbol{k}=(k_{1},\dots,k_{d})\in\operatorname{Ind}_{\mathrm{BZ}}$,
put
\[
F_{\mathrm{BZ}}(\boldsymbol{k})\coloneqq F(\boldsymbol{k};k_{1}-1,\dots,k_{d}-1)=(b_{k_{1}}+b_{k_{1}-1})\cdots(b_{k_{d}}+b_{k_{d}-1}).
\]
The BZ $q$MZV is defined by
\[
\begin{aligned}\zeta^{\mathrm{BZ}}_{q}(\boldsymbol{k}) & \coloneqq\sum_{0<m_{1}<\cdots<m_{d}}\prod^{d}_{i=1}\frac{q^{(k_{i}-1)m_{i}}}{(1-q^{m_{i}})^{k_{i}}}\\
 & =Z^{\mathrm{SZ}}_{q}\left(F_{\mathrm{BZ}}(\boldsymbol{k})\right)
\end{aligned}
\]
for $\boldsymbol{k}\in\operatorname{Ind}_{\mathrm{BZ}}$.

The BZ duality relation \cite{Bradley_qMZV} and the SZ duality relation
\cite{Zhao_qMZV_duality} state that
\[
\zeta^{\mathrm{BZ}}_{q}(\boldsymbol{k})=\zeta^{\mathrm{BZ}}_{q}(\boldsymbol{k}^{\dagger})\qquad\bigl(\boldsymbol{k}\in\operatorname{Ind}_{\mathrm{BZ}}\bigr)
\]
and
\[
\zeta^{\mathrm{SZ}}_{q}(\boldsymbol{k})=\zeta^{\mathrm{SZ}}_{q}(\boldsymbol{k}^{\vee})\qquad\bigl(\boldsymbol{k}\in\operatorname{Ind}_{\mathrm{SZ}}\bigr).
\]
Define
\[
V_{\mathrm{BZ}}\coloneqq\operatorname{span}_{\mathbb{Q}}\left\{ F_{\mathrm{BZ}}(\boldsymbol{k})\ \middle|\ \boldsymbol{k}\in\operatorname{Ind}_{\mathrm{BZ}}\right\} \subset\mathbb{Q}\langle\mathcal{B}\rangle^{0}.
\]
The displayed spanning family is linearly independent, since $b_{\boldsymbol{k}}$
is the unique term of $F_{\mathrm{BZ}}(\boldsymbol{k})$ having maximal
sum of subscripts. Therefore, the following $\mathbb{Q}$-linear maps
are well-defined involutions:
\[
\tau_{\mathrm{BZ}}\colon V_{\mathrm{BZ}}\longrightarrow V_{\mathrm{BZ}},\qquad\tau_{\mathrm{SZ}}\colon\mathbb{Q}\langle\mathcal{B}\rangle^{0}\longrightarrow\mathbb{Q}\langle\mathcal{B}\rangle^{0},
\]
defined by
\[
\tau_{\mathrm{BZ}}\left(F_{\mathrm{BZ}}(\boldsymbol{k})\right)\coloneqq F_{\mathrm{BZ}}(\boldsymbol{k}^{\dagger}),\qquad\tau_{\mathrm{SZ}}(b_{\boldsymbol{k}})\coloneqq b_{\boldsymbol{k}^{\vee}}.
\]
Define the spaces of duality relations by
\[
I^{\mathrm{BZ}}_{\mathrm{dual}}\coloneqq\operatorname{Im}\left(\operatorname{id}_{V_{\mathrm{BZ}}}-\tau_{\mathrm{BZ}}\right)
\]
and
\[
I^{\mathrm{SZ}}_{\mathrm{dual}}\coloneqq\operatorname{Im}\left(\operatorname{id}_{\mathbb{Q}\langle\mathcal{B}\rangle^{0}}-\tau_{\mathrm{SZ}}\right).
\]
The duality relations imply
\[
I^{\mathrm{BZ}}_{\mathrm{dual}}+I^{\mathrm{SZ}}_{\mathrm{dual}}\subset\ker Z^{\mathrm{SZ}}_{q}.
\]

\subsection{\label{subsec:conjectural_duality}A conjectural formula}

Following \cite{Yamamoto_qMPL_duality}, we introduce augmented indices,
writing their components in the order $(\mu,k)$.

An augmented positive integer is a pair
\[
\widetilde{k}=(\mu,k)\in\{0,1\}\times\mathbb{Z}_{\geq1}.
\]
An augmented index is a finite tuple
\[
\widetilde{\boldsymbol{k}}=((\mu_{1},k_{1}),\ldots,(\mu_{d},k_{d}))
\]
of augmented positive integers. It is called admissible if $d=0$
or $(\mu_{d},k_{d})\neq(1,1)$. We denote the set of admissible augmented
indices by $\operatorname{Ind}_{\mathrm{aug}}$.

For $\widetilde{\boldsymbol{k}}\in\operatorname{Ind}_{\mathrm{aug}}$,
put
\[
\omega(\widetilde{\boldsymbol{k}})\coloneqq y_{\mu_{1}}x^{k_{1}-1}\cdots y_{\mu_{d}}x^{k_{d}-1}\in\mathbb{Q}\langle x,y_{0},y_{1}\rangle.
\]
Let $\tau_{\mathrm{aug}}$ be the anti-automorphism of $\mathbb{Q}\langle x,y_{0},y_{1}\rangle$
defined by
\[
\tau_{\mathrm{aug}}(x)\coloneqq y_{1},\qquad\tau_{\mathrm{aug}}(y_{0})\coloneqq y_{0},\qquad\tau_{\mathrm{aug}}(y_{1})\coloneqq x.
\]
For $\widetilde{\boldsymbol{k}}\in\operatorname{Ind}_{\mathrm{aug}}$,
its dual $\widetilde{\boldsymbol{k}}^{\dagger}\in\operatorname{Ind}_{\mathrm{aug}}$
is defined by
\[
\tau_{\mathrm{aug}}\left(\omega(\widetilde{\boldsymbol{k}})\right)=\omega\left(\widetilde{\boldsymbol{k}}^{\dagger}\right).
\]

\begin{defn}
\label{def:theta}For
\[
\widetilde{\boldsymbol{k}}=((\mu_{1},k_{1}),\ldots,(\mu_{d},k_{d}))\in\operatorname{Ind}_{\mathrm{aug}},
\]
define its $\theta$-value by
\[
\theta(\widetilde{\boldsymbol{k}})\coloneqq\sum_{\substack{\varepsilon_{1},\dots,\varepsilon_{d}\in\{0,1\}\\
\varepsilon_{i}=1\ {\rm if}\ \mu_{i}=0
}
}(-1)^{\sum^{d}_{i=1}\mu_{i}\varepsilon_{i}}\mathbb{L}_{k_{1},\ldots,k_{d}}\left(q^{k_{1}-1+\varepsilon_{2}}z^{\varepsilon_{1}-\varepsilon_{2}},\ldots,q^{k_{d}-1+\varepsilon_{d+1}}z^{\varepsilon_{d}-\varepsilon_{d+1}}\right)\in\mathbb{Q}\mathcal{F}^{0},
\]
where $\varepsilon_{d+1}\coloneqq0$. 
\end{defn}

\begin{rem}
Let
\[
\widetilde{\boldsymbol{k}}=((\mu_{1},k_{1}),\ldots,(\mu_{d},k_{d}))\in\operatorname{Ind}_{\mathrm{aug}}
\]
be nonempty, and write
\[
\widetilde{\boldsymbol{k}}_{\mathrm{Y}}\coloneqq((k_{1},\mu_{1}),\ldots,(k_{d},\mu_{d})),\qquad\operatorname{wt}(\widetilde{\boldsymbol{k}})\coloneqq k_{1}+\cdots+k_{d},
\]
where $\widetilde{\boldsymbol{k}}_{\mathrm{Y}}$ is the corresponding
augmented index in Yamamoto's convention. In the notation of Definition
3.1 in \cite{Yamamoto_qMPL_duality}, \Cref{def:theta} gives
\[
\widetilde{\mathrm{Li}}^{(1)}_{q}\left(\widetilde{\boldsymbol{k}}_{\mathrm{Y}};z\right)=(1-q)^{\operatorname{wt}(\widetilde{\boldsymbol{k}})}E\left(\theta(\widetilde{\boldsymbol{k}})\right)(z).
\]
\end{rem}

Then, Yamamoto's $q$-duality \cite[Theorem 3.2]{Yamamoto_qMPL_duality}
can be reformulated as follows.
\begin{thm}[{\cite[Theorem 3.2]{Yamamoto_qMPL_duality}}]
\label{thm:Yamamoto_qdual_by_E}For $\widetilde{\boldsymbol{k}}\in\operatorname{Ind}_{\mathrm{aug}}$,
we have
\[
E\left(\theta(\widetilde{\boldsymbol{k}})\right)=E(\theta(\widetilde{\boldsymbol{k}}^{\dagger})).
\]
\end{thm}

Combining \Cref{prop:evaluation_phi} and \Cref{thm:Yamamoto_qdual_by_E},
we have
\[
\varphi^{\pi}(\theta(\widetilde{\boldsymbol{k}}))-\varphi^{\pi}(\theta(\widetilde{\boldsymbol{k}}^{\dagger}))\in\ker((Z^{\mathrm{SZ}}_{q},L^{\mathrm{SZ}}_{q})),
\]
where $(Z^{\mathrm{SZ}}_{q},L^{\mathrm{SZ}}_{q})\colon\mathbb{Q}\langle\mathcal{B}\rangle^{0}\otimes\mathbb{Q}\langle\mathcal{B}\rangle\to\mathcal{M}(\mathbb{C})$
is the $\mathbb{Q}$-linear map defined by $(Z^{\mathrm{SZ}}_{q},L^{\mathrm{SZ}}_{q})(u\otimes v)\coloneqq Z^{\mathrm{SZ}}_{q}(u)L^{\mathrm{SZ}}_{q}(v)$. 

We now state the main conjecture of this section, which is suggested
by numerical computations.
\begin{conjecture}
\label{conj:phi_dual}For $\widetilde{\boldsymbol{k}}\in\operatorname{Ind}_{\mathrm{aug}}$,
we have
\[
\varphi^{\pi}(\theta(\widetilde{\boldsymbol{k}}))\in V_{\mathrm{BZ}}\otimes\mathbb{Q}\langle\mathcal{B}\rangle.
\]
Furthermore, for $\widetilde{\boldsymbol{k}}\in\operatorname{Ind}_{\mathrm{aug}}$,
we have
\[
\varphi^{\pi}(\theta(\widetilde{\boldsymbol{k}}))=(\tau_{\mathrm{BZ}}\otimes\mathrm{id})(\varphi^{\pi}(\theta(\widetilde{\boldsymbol{k}}^{\dagger}))).
\]
\end{conjecture}

In the next subsection, we prove that the duality relations are contained
in the confluence relations under the assumption that \Cref{conj:phi_dual}
holds whenever $\mu_{i}=1$ for all $i$. These cases can be stated
as follows.
\begin{conjecture}
\label{conj:phi_dual_special}For $\boldsymbol{k}=(k_{1},\ldots,k_{d})\in\operatorname{Ind}_{\mathrm{BZ}}$,
define
\[
\Theta(\boldsymbol{k})\coloneqq\sum_{\varepsilon_{1},\dots,\varepsilon_{d}\in\{0,1\}}(-1)^{\sum^{d}_{i=1}\varepsilon_{i}}\mathbb{L}_{k_{1},\ldots,k_{d}}\left(q^{k_{1}-1+\varepsilon_{2}}z^{\varepsilon_{1}-\varepsilon_{2}},\ldots,q^{k_{d}-1+\varepsilon_{d+1}}z^{\varepsilon_{d}-\varepsilon_{d+1}}\right)\in\mathbb{Q}\mathcal{F}^{0},
\]
where we put $\varepsilon_{d+1}\coloneqq0$. Then,
\[
\varphi^{\pi}(\Theta(\boldsymbol{k}))\in V_{\mathrm{BZ}}\otimes\mathbb{Q}\langle\mathcal{B}\rangle
\]
and
\[
\varphi^{\pi}(\Theta(\boldsymbol{k}))=(\tau_{\mathrm{BZ}}\otimes\mathrm{id})(\varphi^{\pi}(\Theta(\boldsymbol{k}^{\dagger}))).
\]
\end{conjecture}

By the recursive definition of $\varphi$, the contribution of $\mathrm{Const}(\Theta(\boldsymbol{k}))$
is $F_{\mathrm{BZ}}(\boldsymbol{k})\otimes1$, whereas all other contributions
have first tensor factors of smaller weight. Hence, by the triangularity
of the family $F_{\mathrm{BZ}}(\boldsymbol{a})$, the first assertion
is equivalent to

\begin{equation}
\varphi^{\pi}(\Theta(\boldsymbol{k}))=F_{\mathrm{BZ}}(\boldsymbol{k})\otimes1+\sum_{\substack{\boldsymbol{a}\in\operatorname{Ind}_{\mathrm{BZ}}\\
|\boldsymbol{a}|<|\boldsymbol{k}|
}
}F_{\mathrm{BZ}}(\boldsymbol{a})\otimes v_{\boldsymbol{k},\boldsymbol{a}}\label{eq:phi_theta_expansion}
\end{equation}
with $v_{\boldsymbol{k},\boldsymbol{a}}\in\mathbb{Q}\langle\mathcal{B}\rangle$.
By linear independence of the family $F_{\mathrm{BZ}}(\boldsymbol{a})$,
the elements $v_{\boldsymbol{k},\boldsymbol{a}}$ are uniquely determined,
and the second assertion is equivalent to
\begin{equation}
v_{\boldsymbol{k}^{\dagger},\boldsymbol{a}^{\dagger}}=v_{\boldsymbol{k},\boldsymbol{a}}.\label{eq:v_dual}
\end{equation}

\subsection{The confluence relations and the dualities\label{subsec:conf_and_duality}}

In this subsection, we show that \Cref{conj:phi_dual_special} implies
$I^{\mathrm{BZ}}_{\mathrm{dual}}+I^{\mathrm{SZ}}_{\mathrm{dual}}\subset I_{\mathrm{CF}}$.
Let us define a $\mathbb{Q}$-linear map
\[
T\colon\mathbb{Q}\hat{\mathcal{F}}\to\mathbb{Q}\hat{\mathcal{F}}
\]
by $T(u(z))=u(zq)$. Then we have
\[
\Theta(\boldsymbol{k}),T(\Theta(\boldsymbol{k}))\in\mathbb{Q}\mathcal{G}.
\]
We also put
\[
\mathbb{Q}\langle\mathcal{B}\rangle^{!}\coloneqq\mathbb{Q}+\bigoplus^{\infty}_{j=1}\mathbb{Q}\langle\mathcal{B}\rangle(b_{j}+b_{j-1})
\]
and define a $\mathbb{Q}$-linear map 
\[
T_{\mathcal{B}}\colon\mathbb{Q}\langle\mathcal{B}\rangle^{!}\to\mathbb{Q}\langle\mathcal{B}\rangle^{0}
\]
by $T_{\mathcal{B}}(1)=1$ and
\[
T_{\mathcal{B}}(u(b_{j}+b_{j-1}))=ub_{j}\qquad(u\in\mathbb{Q}\langle\mathcal{B}\rangle,\,j\geq1).
\]
The key to the proof is to consider the confluence relations
\begin{equation}
m_{\odot}\circ(\mathrm{id}\otimes\operatorname{reg}_{*})\circ\varphi^{\pi}(\Theta(\boldsymbol{k}))-\pi(\left.\Theta(\boldsymbol{k})\right|_{z=1})\in I^{\mathrm{alg}}_{\mathrm{CF}}\label{eq:conf_theta}
\end{equation}
and
\begin{equation}
m_{\odot}\circ(\mathrm{id}\otimes\operatorname{reg}_{*})\circ\varphi^{\pi}(T(\Theta(\boldsymbol{k})))-\pi(\left.T(\Theta(\boldsymbol{k}))\right|_{z=1})\in I^{\mathrm{alg}}_{\mathrm{CF}}.\label{eq:conf_T_theta}
\end{equation}
The following lemmas evaluate the terms of \eqref{eq:conf_theta}
and \eqref{eq:conf_T_theta}.
\begin{lem}
Let $\boldsymbol{k}\in\operatorname{Ind}_{\mathrm{BZ}}$. Assuming
\eqref{eq:phi_theta_expansion}, we have
\begin{equation}
m_{\odot}\circ(\mathrm{id}\otimes\operatorname{reg}_{*})\circ\varphi^{\pi}(\Theta(\boldsymbol{k}))=F_{\mathrm{BZ}}(\boldsymbol{k})\odot1_{\mathbb{Q}\langle\mathcal{B}\rangle^{0}}+\sum_{\substack{\boldsymbol{a}\in\operatorname{Ind}_{\mathrm{BZ}}\\
|\boldsymbol{a}|<|\boldsymbol{k}|
}
}F_{\mathrm{BZ}}(\boldsymbol{a})\odot\operatorname{reg}_{*}(v_{\boldsymbol{k},\boldsymbol{a}}).\label{eq:conf_theta_1}
\end{equation}
Furthermore, we have $v_{\boldsymbol{k},\boldsymbol{a}}\in\mathbb{Q}\langle\mathcal{B}\rangle^{!}$
and
\begin{equation}
m_{\odot}\circ(\mathrm{id}\otimes\operatorname{reg}_{*})\circ\varphi^{\pi}(T(\Theta(\boldsymbol{k})))=F_{\mathrm{BZ}}(\boldsymbol{k})\odot1_{\mathbb{Q}\langle\mathcal{B}\rangle^{0}}+\sum_{\substack{\boldsymbol{a}\in\operatorname{Ind}_{\mathrm{BZ}}\\
|\boldsymbol{a}|<|\boldsymbol{k}|
}
}F_{\mathrm{BZ}}(\boldsymbol{a})\odot T_{\mathcal{B}}(v_{\boldsymbol{k},\boldsymbol{a}}).\label{eq:conf_Ttheta_1}
\end{equation}
\end{lem}

\begin{proof}
Since
\[
\Theta(\boldsymbol{k})\in\mathbb{Q}(\mathcal{F}^{-1}\cap\mathcal{F}^{0}),\qquad T(\Theta(\boldsymbol{k}))\in\mathbb{Q}(\mathcal{F}^{0}\cap\mathcal{F}^{1})
\]
we have
\[
\varphi(\Theta(\boldsymbol{k}))\in\mathbb{Q}\mathcal{C}\otimes\mathbb{Q}(\mathcal{C}^{-1}_{z}\cap\mathcal{C}^{0}_{z}),\qquad\varphi(T(\Theta(\boldsymbol{k})))\in\mathbb{Q}\mathcal{C}\otimes\mathbb{Q}(\mathcal{C}^{0}_{z}\cap\mathcal{C}^{1}_{z}).
\]
Note that
\[
\pi_{z}(\mathbb{Q}(\mathcal{C}^{-1}_{z}\cap\mathcal{C}^{0}_{z}))=\mathbb{Q}\langle\mathcal{B}\rangle^{!},\qquad\pi_{z}(\mathbb{Q}(\mathcal{C}^{0}_{z}\cap\mathcal{C}^{1}_{z}))=\mathbb{Q}\langle\mathcal{B}\rangle^{0}.
\]
Thus,
\[
\varphi^{\pi}(\Theta(\boldsymbol{k}))\in\mathbb{Q}\langle\mathcal{B}\rangle^{0}\otimes\mathbb{Q}\langle\mathcal{B}\rangle^{!},\qquad\varphi^{\pi}(T(\Theta(\boldsymbol{k})))\in\mathbb{Q}\langle\mathcal{B}\rangle^{0}\otimes\mathbb{Q}\langle\mathcal{B}\rangle^{0}.
\]
The claim $v_{\boldsymbol{k},\boldsymbol{a}}\in\mathbb{Q}\langle\mathcal{B}\rangle^{!}$
follows from this.

Define a $\mathbb{Q}$-linear map $T_{\Omega}\colon\Omega\to\Omega$
by $T_{\Omega}(e_{b})=e_{bq^{-1}}$. Then,
\[
\mathrm{Const}=\mathrm{Const}\circ T,\qquad T\circ S=S\circ(T\otimes T_{\Omega}),\qquad(T\otimes T_{\Omega})\circ\partial=\partial\circ T,
\]
and thus,
\[
\varphi\circ T=(\mathrm{id}_{\mathbb{Q}\mathcal{C}}\otimes T)\circ\varphi
\]
as a map from $\mathbb{Q}(\mathcal{F}^{-1}\cap\mathcal{F}^{0})$ to
$\mathbb{Q}\mathcal{C}\otimes\mathbb{Q}(\mathcal{C}^{0}_{z}\cap\mathcal{C}^{1}_{z})$.
Furthermore, since
\[
\pi_{z}\circ T=T_{\mathcal{B}}\circ\pi_{z}
\]
as a map from $\mathbb{Q}(\mathcal{C}^{-1}_{z}\cap\mathcal{C}^{0}_{z})$
to $\mathbb{Q}\langle\mathcal{B}\rangle^{0}$, we have
\[
\varphi^{\pi}\circ T=(\mathrm{id}_{\mathbb{Q}\langle\mathcal{B}\rangle^{0}}\otimes T_{\mathcal{B}})\circ\varphi^{\pi}
\]
as a map from $\mathbb{Q}(\mathcal{F}^{-1}\cap\mathcal{F}^{0})$ to
$\mathbb{Q}\langle\mathcal{B}\rangle^{0}\otimes\mathbb{Q}\langle\mathcal{B}\rangle^{0}$.
Applying this to \eqref{eq:phi_theta_expansion}, we get
\begin{equation}
\varphi^{\pi}(T(\Theta(\boldsymbol{k})))=F_{\mathrm{BZ}}(\boldsymbol{k})\otimes1+\sum_{\substack{\boldsymbol{a}\in\operatorname{Ind}_{\mathrm{BZ}}\\
|\boldsymbol{a}|<|\boldsymbol{k}|
}
}F_{\mathrm{BZ}}(\boldsymbol{a})\otimes T_{\mathcal{B}}(v_{\boldsymbol{k},\boldsymbol{a}}).\label{eq:phi_Ttheta_expansion}
\end{equation}
Now \eqref{eq:conf_theta_1} and \eqref{eq:conf_Ttheta_1} follow
from \eqref{eq:phi_theta_expansion} and \eqref{eq:phi_Ttheta_expansion},
respectively.
\end{proof}

\begin{lem}
We have
\begin{equation}
\pi(\left.\Theta(\boldsymbol{k})\right|_{z=1})=0\qquad(\boldsymbol{k}\in\operatorname{Ind}_{\mathrm{BZ}}\setminus\{\emptyset\})\label{eq:conf_theta_2}
\end{equation}
and
\begin{equation}
\pi(\left.T(\Theta(\boldsymbol{k}))\right|_{z=1})=b_{\rho(\boldsymbol{k})}\qquad(\boldsymbol{k}\in\operatorname{Ind}_{\mathrm{BZ}}),\label{eq:conf_Ttheta_2}
\end{equation}
where $\rho\colon\operatorname{Ind}_{\mathrm{BZ}}\to\operatorname{Ind}_{\mathrm{SZ}}$
is the map defined by 
\[
\rho((k_{1},\dots,k_{d}))=(k_{1}-1,\dots,k_{d}-1).
\]
\end{lem}

\begin{proof}
For \eqref{eq:conf_theta_2}, the terms with $\varepsilon_{1}=0$
and $1$ cancel pairwise after setting $z=1$. For \eqref{eq:conf_Ttheta_2},
the definitions of $\pi$ and $F$ give
\begin{align*}
\pi(\left.T(\Theta(\boldsymbol{k}))\right|_{z=1}) & =\sum_{\varepsilon_{1},\dots,\varepsilon_{d}\in\{0,1\}}(-1)^{\sum^{d}_{i=1}\varepsilon_{i}}F(k_{1},\dots,k_{d};k_{1}-1+\varepsilon_{1},\dots,k_{d}-1+\varepsilon_{d})\\
 & =b_{\rho(\boldsymbol{k})}.\qedhere
\end{align*}
\end{proof}

\begin{thm}
\label{thm:BZ_dual_conf}If \Cref{conj:phi_dual_special} is true,
then
\[
I^{\mathrm{BZ}}_{\mathrm{dual}}\subset I_{\mathrm{CF}}.
\]
\end{thm}

\begin{proof}
We prove
\[
F_{\mathrm{BZ}}(\boldsymbol{k})-F_{\mathrm{BZ}}(\boldsymbol{k}^{\dagger})\in I_{\mathrm{CF}}
\]
for $\boldsymbol{k}\in\operatorname{Ind}_{\mathrm{BZ}}$ by induction
on $|\boldsymbol{k}|\coloneqq k_{1}+\cdots+k_{d}$. The case $\boldsymbol{k}=\emptyset$
is trivial. Hence, assume that $\boldsymbol{k}\neq\emptyset$. By
\eqref{eq:conf_theta}, \eqref{eq:conf_theta_1}, and \eqref{eq:conf_theta_2},
we have
\[
F_{\mathrm{BZ}}(\boldsymbol{k})\equiv-\sum_{\substack{\boldsymbol{a}\in\operatorname{Ind}_{\mathrm{BZ}}\\
|\boldsymbol{a}|<|\boldsymbol{k}|
}
}F_{\mathrm{BZ}}(\boldsymbol{a})\odot\operatorname{reg}_{*}(v_{\boldsymbol{k},\boldsymbol{a}})\pmod{I^{\mathrm{alg}}_{\mathrm{CF}}}.
\]
Thus, 
\begin{align*}
F_{\mathrm{BZ}}(\boldsymbol{k})-F_{\mathrm{BZ}}(\boldsymbol{k}^{\dagger}) & \equiv-\sum_{\substack{\boldsymbol{a}\in\operatorname{Ind}_{\mathrm{BZ}}\\
|\boldsymbol{a}|<|\boldsymbol{k}|
}
}F_{\mathrm{BZ}}(\boldsymbol{a})\odot\operatorname{reg}_{*}(v_{\boldsymbol{k},\boldsymbol{a}})\\
 & \quad+\sum_{\substack{\boldsymbol{a}'\in\operatorname{Ind}_{\mathrm{BZ}}\\
|\boldsymbol{a}'|<|\boldsymbol{k}|
}
}F_{\mathrm{BZ}}(\boldsymbol{a}')\odot\operatorname{reg}_{*}(v_{\boldsymbol{k}^{\dagger},\boldsymbol{a}'})\\
 & =-\sum_{\substack{\boldsymbol{a}\in\operatorname{Ind}_{\mathrm{BZ}}\\
|\boldsymbol{a}|<|\boldsymbol{k}|
}
}(F_{\mathrm{BZ}}(\boldsymbol{a})-F_{\mathrm{BZ}}(\boldsymbol{a}^{\dagger}))\odot\operatorname{reg}_{*}(v_{\boldsymbol{k},\boldsymbol{a}})\qquad(\text{by \eqref{eq:v_dual}})\\
 & \equiv0\qquad(\text{by induction hypothesis}).\qedhere
\end{align*}
\end{proof}

\begin{thm}
If \Cref{conj:phi_dual_special} is true, then
\[
I^{\mathrm{SZ}}_{\mathrm{dual}}\subset I_{\mathrm{CF}}.
\]
\end{thm}

\begin{proof}
Let $\boldsymbol{k}\in\operatorname{Ind}_{\mathrm{BZ}}$. By \eqref{eq:conf_T_theta},
\eqref{eq:conf_Ttheta_1}, and \eqref{eq:conf_Ttheta_2}, we have
\[
b_{\rho(\boldsymbol{k})}\equiv F_{\mathrm{BZ}}(\boldsymbol{k})+\sum_{\substack{\boldsymbol{a}\in\operatorname{Ind}_{\mathrm{BZ}}\\
|\boldsymbol{a}|<|\boldsymbol{k}|
}
}F_{\mathrm{BZ}}(\boldsymbol{a})\odot T_{\mathcal{B}}(v_{\boldsymbol{k},\boldsymbol{a}})\pmod{I^{\mathrm{alg}}_{\mathrm{CF}}}.
\]
Thus, 
\begin{align*}
b_{\rho(\boldsymbol{k})}-b_{\rho(\boldsymbol{k}^{\dagger})} & \equiv F_{\mathrm{BZ}}(\boldsymbol{k})+\sum_{\substack{\boldsymbol{a}\in\operatorname{Ind}_{\mathrm{BZ}}\\
|\boldsymbol{a}|<|\boldsymbol{k}|
}
}F_{\mathrm{BZ}}(\boldsymbol{a})\odot T_{\mathcal{B}}(v_{\boldsymbol{k},\boldsymbol{a}})\\
 & \quad-F_{\mathrm{BZ}}(\boldsymbol{k}^{\dagger})-\sum_{\substack{\boldsymbol{a}'\in\operatorname{Ind}_{\mathrm{BZ}}\\
|\boldsymbol{a}'|<|\boldsymbol{k}|
}
}F_{\mathrm{BZ}}(\boldsymbol{a}')\odot T_{\mathcal{B}}(v_{\boldsymbol{k}^{\dagger},\boldsymbol{a}'})\\
 & =\left(F_{\mathrm{BZ}}(\boldsymbol{k})-F_{\mathrm{BZ}}(\boldsymbol{k}^{\dagger})\right)+\sum_{\substack{\boldsymbol{a}\in\operatorname{Ind}_{\mathrm{BZ}}\\
|\boldsymbol{a}|<|\boldsymbol{k}|
}
}(F_{\mathrm{BZ}}(\boldsymbol{a})-F_{\mathrm{BZ}}(\boldsymbol{a}^{\dagger}))\odot T_{\mathcal{B}}(v_{\boldsymbol{k},\boldsymbol{a}})\qquad(\text{by \eqref{eq:v_dual}})\\
 & \equiv0\qquad(\text{by \cref{thm:BZ_dual_conf}})
\end{align*}
modulo $I^{\mathrm{alg}}_{\mathrm{CF}}$. Since $\rho(\operatorname{Ind}_{\mathrm{BZ}})=\operatorname{Ind}_{\mathrm{SZ}}$
and $\rho(\boldsymbol{k}^{\dagger})=\rho(\boldsymbol{k})^{\vee}$,
this completes the proof.
\end{proof}

\section{\label{sec:Further_questions}Further questions}

We conclude by recording several questions arising from our construction.
\begin{itemize}
\item \textbf{Alternative paths.} The construction is based on the standard
path from $u(z)$ to $u(zq)$. It is natural to ask what relations
are obtained by using other paths in place of the standard path. We
expect that the analogue of $\varphi^{\pi}(u)$ obtained from such
a path is independent of the chosen path.
\item \textbf{The duality conjecture.} It remains to prove \Cref{conj:phi_dual}.
\item \textbf{Classical confluence relations.} The relationship between
the confluence relations constructed in this paper and the classical
confluence relations remains to be investigated in more detail.
\item \textbf{The pentagon equation.} In the classical setting, Furusho
proved that Drinfeld\textquoteright s pentagon equation is equivalent
to the confluence relations for classical MZVs \cite{Furusho_pentagon_confluence}.
This raises the problem of formulating a $q$-analogue of the pentagon
equation for $q$MZVs and establishing its equivalence with the confluence
relations constructed in this paper.
\item \textbf{Multiple Eisenstein series.} Another direction is to construct
confluence relations for multiple Eisenstein series.
\end{itemize}

\appendix

\section{Summary of notation for spaces of $q$MPLs\label{app:notation-summary}}

\[
\hat{\mathcal{F}}=\left\{ \mathbb{L}_{\boldsymbol{k}}(\boldsymbol{x})\,\middle|\,\begin{array}{l}
d\geq0,\quad k_{i}\in\mathbb{Z}_{\geq1},\\
x_{i}\in\{z^{h}q^{m}\mid h\in\{0,\pm1\},m\in\mathbb{Z}\},\\
x_{i}\cdots x_{d}\in zq^{\mathbb{Z}}\cup q^{\mathbb{Z}_{>0}}\quad(1\leq i\leq d)
\end{array}\right\} ,
\]

\[
\mathcal{F}^{n}=\{\mathbb{L}_{\boldsymbol{k}}(\boldsymbol{x})\in\hat{\mathcal{F}}\,\mid\,x_{i}\in\{vq^{s}\,\mid\,v\in\{1,zq^{n},(zq^{n})^{-1}\},\,0\leq s\leq k_{i}\}\},
\]
\[
\mathcal{G}=\{\mathbb{L}_{\boldsymbol{k}}(\boldsymbol{x})\in\mathcal{F}^{0}\,\mid\,d=0\text{ or }x_{d}\neq z\},
\]
\[
\mathcal{F}=\bigcup_{n\in\mathbb{Z}}\mathcal{F}^{n},
\]
\[
\mathcal{C}=\{u\in\mathcal{F}\,\mid\,u\text{ is independent of }z\},
\]
\[
\hat{\mathcal{C}}_{z}=\{\mathbb{L}_{\boldsymbol{k}}(\boldsymbol{x})\in\hat{\mathcal{F}}\,\mid\,x_{i}\in q^{\mathbb{Z}}\quad(1\leq i<d),\,x_{d}\in zq^{\mathbb{Z}}\}\cup\{\mathbb{L}(\emptyset)\},
\]
\[
\mathcal{C}^{n}_{z}=\left\{ \mathbb{L}_{k_{1},\dots,k_{d}}(q^{n_{1}},\dots,q^{n_{d-1}},zq^{n+n_{d}})\,\middle|\,0\leq n_{i}\leq k_{i}\ (1\leq i\leq d)\right\} \cup\{\mathbb{L}(\emptyset)\}.
\]
Their inclusion relations are summarized as follows:
\[
\mathcal{C}\subset\mathcal{F}^{n}\subset\mathcal{F}\subset\hat{\mathcal{F}},\qquad\mathcal{G}\subset\mathcal{F}^{0},\qquad\mathcal{C}^{n}_{z}\subset\hat{\mathcal{C}}_{z}\subset\hat{\mathcal{F}}.
\]

\section{\label{app:Proof_of_some_claims}Proofs of some claims}

\subsection{\label{subsec:Proof_of_stuffle_residue}Proof of \Cref{thm:residue_and_stuffle_regularization}}
\begin{lem}
\label{lem:L_st_b0}For $u\in\mathbb{Q}\langle\mathcal{B}\rangle$,
we have
\[
L^{\mathrm{SZ}}_{q}(u*b_{0})(z)=z\frac{d}{dz}L^{\mathrm{SZ}}_{q}(u)(z)+\frac{z}{1-z}L^{\mathrm{SZ}}_{q}(u)(z).
\]
\end{lem}

\begin{proof}
By $\mathbb{Q}$-linearity, it suffices to consider the case
\[
u=b_{k_{1}}\cdots b_{k_{r}}.
\]
Then, we have
\[
u*b_{0}=\sum^{r}_{i=0}b_{k_{1}}\cdots b_{k_{i}}b_{0}b_{k_{i+1}}\cdots b_{k_{r}}+ru.
\]
Since
\[
L^{\mathrm{SZ}}_{q}(b_{k_{1}}\cdots b_{k_{i}}b_{0}b_{k_{i+1}}\cdots b_{k_{r}})(z)=\sum_{0=m_{0}<m_{1}<\cdots<m_{r}}(m_{i+1}-m_{i}-1)\prod^{r}_{j=1}\left(\frac{q^{m_{j}}}{1-q^{m_{j}}}\right)^{k_{j}}z^{m_{r}}
\]
for $0\leq i<r$ and
\[
L^{\mathrm{SZ}}_{q}(b_{k_{1}}\cdots b_{k_{r}}b_{0})(z)=\frac{z}{1-z}L^{\mathrm{SZ}}_{q}(b_{k_{1}}\cdots b_{k_{r}})(z),
\]
we have
\begin{align*}
L^{\mathrm{SZ}}_{q}(u*b_{0})(z) & =\sum_{0=m_{0}<m_{1}<\cdots<m_{r}}\left(\sum^{r-1}_{i=0}(m_{i+1}-m_{i}-1)+r\right)\prod^{r}_{j=1}\left(\frac{q^{m_{j}}}{1-q^{m_{j}}}\right)^{k_{j}}z^{m_{r}}+\frac{z}{1-z}L^{\mathrm{SZ}}_{q}(u)(z)\\
 & =z\frac{d}{dz}L^{\mathrm{SZ}}_{q}(u)(z)+\frac{z}{1-z}L^{\mathrm{SZ}}_{q}(u)(z).\qedhere
\end{align*}
\end{proof}

\begin{lem}
For $u\in\mathbb{Q}\langle\mathcal{B}\rangle$, we have
\[
\operatorname*{Res}_{z=1}\frac{L^{\mathrm{SZ}}_{q}(u*b_{0})(z)}{z(z-1)}\,dz=0.
\]
\end{lem}

\begin{proof}
By \Cref{lem:L_st_b0}, we have
\begin{align*}
\frac{L^{\mathrm{SZ}}_{q}(u*b_{0})(z)}{z(z-1)} & =\frac{d}{dz}\left(\frac{L^{\mathrm{SZ}}_{q}(u)(z)}{z-1}\right).
\end{align*}
Since the residue of the derivative of a Laurent series is zero, the
residue vanishes.
\end{proof}

\begin{lem}
For $u\in\mathbb{Q}\langle\mathcal{B}\rangle^{0}$, we have
\[
\operatorname*{Res}_{z=1}\frac{L^{\mathrm{SZ}}_{q}(u)(z)}{z(z-1)}\,dz=Z^{\mathrm{SZ}}_{q}(u).
\]
\end{lem}

\begin{proof}
Since $L^{\mathrm{SZ}}_{q}(u)$ is holomorphic at $z=1$ and $L^{\mathrm{SZ}}_{q}(u)(1)=Z^{\mathrm{SZ}}_{q}(u)$,
the assertion holds.
\end{proof}

We now prove the theorem.

\begin{proof}[Proof of \Cref{thm:residue_and_stuffle_regularization}]
Let $u\in\mathbb{Q}\langle\mathcal{B}\rangle$. By the stuffle-regularization
decomposition, write
\[
u=u_{0}+\sum^{k}_{j=1}u_{j}*b^{*j}_{0},\qquad u_{0},\ldots,u_{k}\in\mathbb{Q}\langle\mathcal{B}\rangle^{0}.
\]
For $j\geq1$, we have
\[
u_{j}*b^{*j}_{0}=\left(u_{j}*b^{*(j-1)}_{0}\right)*b_{0}.
\]
Hence, by linearity and the two preceding lemmas,
\[
\begin{aligned}\operatorname*{Res}_{z=1}\frac{L^{\mathrm{SZ}}_{q}(u)(z)}{z(z-1)}\,dz & =\operatorname*{Res}_{z=1}\frac{L^{\mathrm{SZ}}_{q}(u_{0})(z)}{z(z-1)}\,dz\\
 & \quad+\sum^{k}_{j=1}\operatorname*{Res}_{z=1}\frac{L^{\mathrm{SZ}}_{q}\left(\left(u_{j}*b^{*(j-1)}_{0}\right)*b_{0}\right)(z)}{z(z-1)}\,dz\\
 & =Z^{\mathrm{SZ}}_{q}(u_{0})\\
 & =Z^{\mathrm{SZ}}_{q}\!\left(\operatorname{reg}_{*}(u)\right).
\end{aligned}
\]
This proves the theorem.
\end{proof}

\subsection{\label{subsec:Proof_of_fn_standard}Proof of \Cref{thm:Fn_StandardPath}}
\begin{lem}
\label{lem:std_path_qmpls}Let $n\in\mathbb{Z}$ and $u\in\mathcal{F}^{n}$.
Let 
\[
\operatorname{Std}(u)=(u_{0},u_{1},\dots,u_{l})
\]
be the standard path from $u(z)$ to $u(zq)$. Then $u_{0},\dots,u_{l}$
are formal $q$MPLs.
\end{lem}

\begin{proof}
Write
\[
u_{i}=\mathbb{L}_{k_{1},\ldots,k_{d}}(x^{(i)}_{1},\ldots,x^{(i)}_{d})\qquad(0\leq i\leq l),
\]
and
\[
u=\mathbb{L}_{k_{1},\ldots,k_{d}}(x_{1},\ldots,x_{d})=\mathbb{L}_{k_{1},\ldots,k_{d}}(z^{h_{1}}q^{m_{1}},\ldots,z^{h_{d}}q^{m_{d}}).
\]
We prove that each $u_{i}$ is a formal $q$MPL by contradiction.
Assume that $u_{i}$ is not a formal $q$MPL, i.e., there exists $p\in\{1,\dots,d\}$
such that
\begin{equation}
x^{(i)}_{p}\cdots x^{(i)}_{d}\in q^{\mathbb{Z}_{\leq0}}.\label{eq:assumption_prod_in}
\end{equation}
Since $u\in\mathcal{F}^{n}$, a $z$-dependent argument has one of
the forms
\[
x_{j}=zq^{n+s_{j}}\qquad\text{or}\qquad x_{j}=z^{-1}q^{-n+s_{j}}\qquad(0\leq s_{j}\leq k_{j}).
\]
For a $z$-dependent position $j$, let $r_{j}$ denote its level.
Then, by definition of the level,
\begin{equation}
r_{j}=n+s_{j}\quad(h_{j}=1),\qquad r_{j}=n+1-s_{j}\qquad(h_{j}=-1).\label{eq:rj_sj}
\end{equation}

Let
\begin{align*}
A & \coloneqq\{a\in\{p,\dots,d\}\,\mid\,h_{a}=-1\},\\
B & \coloneqq\{b\in\{p,\dots,d\}\,\mid\,h_{b}=1\}.
\end{align*}
Then, by assumption, $\#A=\#B$ and
\[
A=\{a_{1}<\cdots<a_{T}\},B=\{b_{1}<\cdots<b_{T}\}
\]
with 
\[
p\leq a_{1}<b_{1}<a_{2}<b_{2}<\cdots<a_{T}<b_{T}\leq d.
\]

We first establish the following two claims:
\begin{equation}
r_{a_{t}}<r_{b_{t}}\Longrightarrow x^{(i)}_{a_{t}}x^{(i)}_{b_{t}}\in q^{\mathbb{Z}_{\geq1}},\label{eq:xat_xbt_geq_1}
\end{equation}
\begin{equation}
x^{(i)}_{a_{t}}x^{(i)}_{b_{t}}\in q^{\mathbb{Z}_{\geq0}},\label{eq:xat_xbt_nonnegative}
\end{equation}
for $t\in\{1,\dots,T\}$. We begin with \eqref{eq:xat_xbt_geq_1}.
By \eqref{eq:rj_sj}, the assumption $r_{a_{t}}<r_{b_{t}}$ is equivalent
to
\[
s_{a_{t}}+s_{b_{t}}\geq2.
\]
Thus,
\[
x^{(i)}_{a_{t}}x^{(i)}_{b_{t}}=x_{a_{t}}x_{b_{t}}\cdot\frac{x^{(i)}_{a_{t}}}{x_{a_{t}}}\cdot\frac{x^{(i)}_{b_{t}}}{x_{b_{t}}}\in q^{\mathbb{Z}_{\geq2}}q^{\{-1,0\}}q^{\{0,1\}}=q^{\mathbb{Z}_{\geq1}}.
\]
We next prove \eqref{eq:xat_xbt_nonnegative}. Since the case $r_{a_{t}}<r_{b_{t}}$
follows from \eqref{eq:xat_xbt_geq_1}, we assume that $r_{a_{t}}\geq r_{b_{t}}$.
Then
\[
x_{a_{t}}x_{b_{t}}=q^{s_{a_{t}}+s_{b_{t}}}\in q^{\mathbb{Z}_{\geq0}}.
\]
Furthermore, since the $b_{t}$-th position is replaced before the
$a_{t}$-th position, we have
\[
\frac{x^{(i)}_{a_{t}}}{x_{a_{t}}}\cdot\frac{x^{(i)}_{b_{t}}}{x_{b_{t}}}\in q^{\{0,1\}}.
\]
Thus, \eqref{eq:xat_xbt_nonnegative} holds.

Then, by \eqref{eq:assumption_prod_in}, \eqref{eq:xat_xbt_geq_1}
and \eqref{eq:xat_xbt_nonnegative}, for all $t\in\{1,\dots,T\}$,
$r_{a_{t}}\geq r_{b_{t}}$ and thus, 
\[
\frac{x^{(i)}_{a_{t}}}{x_{a_{t}}}\cdot\frac{x^{(i)}_{b_{t}}}{x_{b_{t}}}\in q^{\{0,1\}}.
\]
Therefore, 
\[
x^{(i)}_{p}\cdots x^{(i)}_{d}=x_{p}\cdots x_{d}\prod^{T}_{t=1}\frac{x^{(i)}_{a_{t}}}{x_{a_{t}}}\cdot\frac{x^{(i)}_{b_{t}}}{x_{b_{t}}}\in x_{p}\cdots x_{d}q^{\mathbb{Z}_{\geq0}}.
\]
Since $x_{p}\cdots x_{d}\in q^{\mathbb{Z}_{>0}}$ by the assumption
$u\in\hat{\mathcal{F}}$, this contradicts \eqref{eq:assumption_prod_in}.
This completes the proof that $u_{0},\dots,u_{l}$ are formal $q$MPLs.
\end{proof}

\begin{lem}
\label{lem:exists_ldash}Let $n\in\mathbb{Z}$ and $u\in\mathcal{F}^{n}$.
Let 
\[
\operatorname{Std}(u)=(u_{0},u_{1},\dots,u_{l})
\]
be the standard path from $u(z)$ to $u(zq)$. Then there exists $0\leq l'\leq l$
such that
\[
u_{1},\dots,u_{l'}\in\mathcal{F}^{n}\qquad\text{and}\qquad u_{l'},\dots,u_{l}\in\mathcal{F}^{n+1}.
\]
In other words, if $u_{i}\notin\mathcal{F}^{n+1}$ and $u_{i'}\notin\mathcal{F}^{n}$,
then $i+1<i'$.
\end{lem}

\begin{proof}
We use the same symbols as in the proof of \Cref{lem:std_path_qmpls}.
Assume that $u_{i}\notin\mathcal{F}^{n+1}$ and $u_{i'}\notin\mathcal{F}^{n}$.
By \Cref{lem:std_path_qmpls}, $u_{i}$ and $u_{i'}$ are formal $q$MPLs.
Hence their failure to belong to $\mathcal{F}^{n+1}$ and $\mathcal{F}^{n}$,
respectively, must come from the restrictions on their arguments in
the definitions of these sets.

Since $u_{i'}\notin\mathcal{F}^{n}$, there exists $a\in\{1,\dots,d\}$
such that
\[
x^{(i')}_{a}(z)=x_{a}(zq)\neq x_{a}(z)
\]
and
\[
x_{a}=zq^{n+k_{a}}\text{ or }x_{a}=z^{-1}q^{-n}.
\]
Then, the level $r_{a}$ is $n+k_{a}$ or $n+1$, and so, $r_{a}\geq n+1$.

Similarly, since $u_{i}\notin\mathcal{F}^{n+1}$, there exists $b\in\{1,\dots,d\}$
such that
\[
x^{(i)}_{b}(z)=x_{b}(z)\neq x_{b}(zq)
\]
and
\[
x_{b}=zq^{n}\text{ or }x_{b}=z^{-1}q^{-n+k_{b}}.
\]
Then, the level $r_{b}$ is $n$ or $n+1-k_{b}$, and so, $r_{b}\leq n$.

Therefore, $r_{b}<r_{a}$. By the rule that we replace entries in
increasing order, we have $i+1<i'$. 
\end{proof}

\begin{lem}
\label{lem:D_adjacent_fn}Let $n\in\mathbb{Z}$ and $v,v'\in\mathcal{F}^{n}$.
If $v$ and $v'$ are adjacent, then
\[
D(v,v')\in\mathbb{Q}\mathcal{F}^{n}\otimes(\mathbb{Q}e_{0}\oplus\mathbb{Q}e_{q^{-n}}).
\]
\end{lem}

\begin{proof}
It follows from the definition of $D$.
\end{proof}

\begin{thm*}[\Cref{thm:Fn_StandardPath}]
Let $n\in\mathbb{Z}$ and $u\in\mathcal{F}^{n}$. Let 
\[
\operatorname{Std}(u)=(u_{0},u_{1},\dots,u_{l})
\]
be the standard path from $u(z)$ to $u(zq)$. Then $u_{0},\dots,u_{l}$
are formal $q$MPLs. Furthermore, for $i=1,\dots,l$, the elements
$u_{i-1}$ and $u_{i}$ are adjacent and 
\[
D(u_{i-1},u_{i})\in\mathbb{Q}\mathcal{F}^{n}\otimes(\mathbb{Q}e_{0}\oplus\mathbb{Q}e_{q^{-n}})+\mathbb{Q}\mathcal{F}^{n+1}\otimes(\mathbb{Q}e_{0}\oplus\mathbb{Q}e_{q^{-(n+1)}}).
\]
\end{thm*}
\begin{proof}
The first claim is just \Cref{lem:std_path_qmpls}. The second claim
follows from \Cref{lem:exists_ldash} and \Cref{lem:D_adjacent_fn}.
\end{proof}

\subsection{\label{subsec:proof_der_and_stuffle}Proof of \Cref{lem:der_and_stuffle}}
\begin{lem*}[\Cref{lem:der_and_stuffle}]
For $u\in\mathbb{Q}\mathcal{C}$, we have the following identity
of maps from $\mathbb{Q}(\mathcal{F}\cap\hat{\mathcal{C}}_{z})$ to
$\mathbb{Q}\mathcal{F}\otimes\Omega$:
\[
\partial\circ l^{*}_{u}=(l^{*}_{u}\otimes\mathrm{id}_{\Omega})\circ\partial.
\]
\end{lem*}
\begin{proof}
By $\mathbb{Q}$-bilinearity, it suffices to assume that $u$ and
$v$ are basis elements; that is,
\[
u\in\mathcal{C},\qquad v\in\mathcal{F}\cap\hat{\mathcal{C}}_{z}.
\]
 The assertion is immediate if $u$ or $v$ is the empty word, so
suppose that both are nonempty. Let
\[
u=\mathbb{L}_{k_{1},\ldots,k_{r}}(x_{1},\ldots,x_{r})\in\mathcal{C}
\]
and
\[
v=\mathbb{L}_{l_{1},\ldots,l_{s}}(y_{1},\ldots,y_{s})\in\mathcal{F}\cap\hat{\mathcal{C}}_{z}
\]
with $y_{s}=zq^{m}$. Every stuffle summand of $u*v$ contains exactly
one argument depending on $z$, and hence its standard path consists
of a single step. First suppose that $l_{s}>1$. Then
\[
\partial v=\mathbb{L}_{l_{1},\ldots,l_{s-1},l_{s}-1}(y_{1},\ldots,y_{s-1},zq^{m})\otimes e_{0}.
\]
For each stuffle datum contributing to $u*v$, applying $\partial$
decreases the weight of the unique $z$-dependent entry by one. If
the last entry of $v$ is not merged, this replaces $l_{s}$ by $l_{s}-1$;
if it is merged with an entry of weight $k$, this replaces $k+l_{s}$
by $k+l_{s}-1$. Thus, term by term,
\[
\partial(u*v)=(l^{*}_{u}\otimes\operatorname{id}_{\Omega})(\partial v).
\]
Suppose next that $l_{s}=1$. By the definition of $\partial$,
\[
\partial v=-\mathbb{L}_{l_{1},\ldots,l_{s-1}}(y_{1},\ldots,y_{s-1}y_{s})\otimes e_{q^{-m}},
\]
where the formal $q$MPL on the right-hand side is understood as $\mathbb{L}(\emptyset)$
when $s=1$. By definition,
\[
u*v=\sum^{r+s}_{t=\max(r,s)}\sum_{(f,g)\in\mathrm{Stu}_{t}(r,s)}\mathbb{L}_{m_{1},\dots,m_{t}}(w_{1},\dots,w_{t}),
\]
where
\[
(m_{i},w_{i})=(m^{(f,g)}_{i},w^{(f,g)}_{i})\coloneqq\begin{cases}
(k_{j},x_{j}) & \text{if }f^{-1}(\{i\})=\{j\}\text{ and }g^{-1}(\{i\})=\emptyset\\
(l_{j},y_{j}) & \text{if }f^{-1}(\{i\})=\emptyset\text{ and }g^{-1}(\{i\})=\{j\}\\
(k_{j}+l_{j'},x_{j}y_{j'}) & \text{if }f^{-1}(\{i\})=\{j\}\text{ and }g^{-1}(\{i\})=\{j'\}.
\end{cases}
\]
Here, the unique $z$-dependent entry appears at $i=g(s)$, and $m_{g(s)}$
is equal to $1$ if and only if $f^{-1}(\{g(s)\})=\emptyset$. When
$f^{-1}(\{g(s)\})=\emptyset$, we have
\begin{align*}
 & \partial\left(\mathbb{L}_{m_{1},\dots,m_{t}}(w_{1},\dots,w_{t})\right)\\
 & =-\mathbb{L}_{m_{1},\dots,\widehat{m_{g(s)}},\dots,m_{t}}(w_{1},\dots,w_{g(s)-1}w_{g(s)},\dots,w_{t})\otimes e_{q^{-m}}\\
 & \quad+\begin{cases}
\mathbb{L}_{m_{1},\dots,\widehat{m_{g(s)}},\dots,m_{t}}(w_{1},\dots,w_{g(s)}w_{g(s)+1},\dots,w_{t})\otimes(e_{q^{-m}}-e_{0}) & \text{if }g(s)<t\\
0 & \text{if }g(s)=t,
\end{cases}
\end{align*}
where we understand the first term as 
\[
-\mathbb{L}_{m_{2},\dots,m_{t}}(w_{2},\dots,w_{t})\otimes e_{q^{-m}}
\]
when $g(s)=1$. Thus
\[
\partial(u*v)=-T_{1}\otimes e_{q^{-m}}+T_{2}\otimes(e_{q^{-m}}-e_{0})+T_{3}\otimes e_{0},
\]
where
\begin{align*}
T_{1} & \coloneqq\sum^{r+s}_{t=\max(r,s)}\sum_{\substack{(f,g)\in\mathrm{Stu}_{t}(r,s)\\
f^{-1}(\{g(s)\})=\emptyset
}
}\mathbb{L}_{m_{1},\dots,\widehat{m_{g(s)}},\dots,m_{t}}(w_{1},\dots,w_{g(s)-1}w_{g(s)},\dots,w_{t}),\\
T_{2} & \coloneqq\sum^{r+s}_{t=\max(r,s)}\sum_{\substack{(f,g)\in\mathrm{Stu}_{t}(r,s)\\
f^{-1}(\{g(s)\})=\emptyset\\
g(s)<t
}
}\mathbb{L}_{m_{1},\dots,\widehat{m_{g(s)}},\dots,m_{t}}(w_{1},\dots,w_{g(s)}w_{g(s)+1},\dots,w_{t}),\\
T_{3} & \coloneqq\sum^{r+s}_{t=\max(r,s)}\sum_{\substack{(f,g)\in\mathrm{Stu}_{t}(r,s)\\
f^{-1}(\{g(s)\})\neq\emptyset
}
}\mathbb{L}_{m_{1},\dots,m_{g(s)}-1,\dots,m_{t}}(w_{1},\dots,w_{t}).
\end{align*}
For $\alpha=1,\dots,r$, we put
\[
U(\alpha)\coloneqq\sum^{\alpha+s-2}_{t=\max(\alpha-1,s-1)}\sum_{(f,g)\in\mathrm{Stu}_{t}(\alpha-1,s-1)}\mathbb{L}_{m_{1},\dots,m_{t},k_{\alpha},\dots,k_{r}}(w_{1},\dots,w_{t},x_{\alpha}zq^{m},x_{\alpha+1},\dots,x_{r}),
\]
where $(m_{i},w_{i})$ are defined from $(f,g)$ by the same rule
as above. Roughly speaking, $U(\alpha)$ is obtained by appending
\[
\mathbb{L}_{k_{\alpha},\dots,k_{r}}(x_{\alpha}zq^{m},x_{\alpha+1},\dots,x_{r})
\]
 to each summand of
\[
\mathbb{L}_{k_{1},\dots,k_{\alpha-1}}(x_{1},\dots,x_{\alpha-1})*\mathbb{L}_{l_{1},\ldots,l_{s-1}}(y_{1},\ldots,y_{s-1}).
\]
For $T_{1}$, we have
\begin{align*}
T_{1} & =\sum^{r+s}_{t=\max(r,s)}\sum_{\substack{(f,g)\in\mathrm{Stu}_{t}(r,s)\\
f^{-1}(\{g(s)\})=\emptyset\\
g(s)=1\text{ or }g^{-1}(\{g(s)-1\})\neq\emptyset
}
}\mathbb{L}_{m_{1},\dots,\widehat{m_{g(s)}},\dots,m_{t}}(w_{1},\dots,w_{g(s)-1}w_{g(s)},\dots,w_{t})\\
 & \quad+\sum^{r+s}_{t=\max(r,s)}\sum_{\substack{(f,g)\in\mathrm{Stu}_{t}(r,s)\\
f^{-1}(\{g(s)\})=\emptyset\\
g(s)>1,\,g^{-1}(\{g(s)-1\})=\emptyset
}
}\mathbb{L}_{m_{1},\dots,\widehat{m_{g(s)}},\dots,m_{t}}(w_{1},\dots,w_{g(s)-1}w_{g(s)},\dots,w_{t})
\end{align*}
Here, the first sum is equal to 
\[
u*\mathbb{L}_{l_{1},\ldots,l_{s-1}}(y_{1},\ldots,y_{s-1}y_{s}),
\]
while the second sum is equal to
\[
\sum^{r}_{\alpha=1}U(\alpha),
\]
as can be seen by taking $\alpha$ such that $f(\alpha)=g(s)-1$.
Thus,
\[
T_{1}=u*\mathbb{L}_{l_{1},\ldots,l_{s-1}}(y_{1},\ldots,y_{s-1}y_{s})+\sum^{r}_{\alpha=1}U(\alpha).
\]
For $T_{2}$, taking $\alpha$ such that $f(\alpha)=g(s)+1$, we obtain
\[
T_{2}=\sum^{r}_{\alpha=1}U(\alpha).
\]
For $T_{3}$, taking $\alpha$ such that $f(\alpha)=g(s)$, we obtain
\[
T_{3}=\sum^{r}_{\alpha=1}U(\alpha).
\]
Hence,
\begin{align*}
\partial(u*v) & =-T_{1}\otimes e_{q^{-m}}+T_{2}\otimes(e_{q^{-m}}-e_{0})+T_{3}\otimes e_{0}\\
 & =-u*\mathbb{L}_{l_{1},\ldots,l_{s-1}}(y_{1},\ldots,y_{s-1}y_{s})\otimes e_{q^{-m}}\\
 & =(l^{*}_{u}\otimes\operatorname{id}_{\Omega})(\partial v).
\end{align*}
This completes the proof.
\end{proof}

\section{\label{app:Tables}Tables related to \Cref{conj:phi_dual}}

\begin{longtable}[c]{>{\raggedright}p{0.14\linewidth}>{\raggedright}p{0.14\linewidth}>{\raggedright}p{0.63\linewidth}}
\caption{Expansions of $\varphi^{\pi}(\theta(\widetilde{\boldsymbol{k}}))$
for all admissible augmented indices of weight at most 3 and selected
higher-weight examples. The following notation is used in this table:
$\hat{k}\protect\coloneqq(0,k)$, $k\protect\coloneqq(1,k)$, $F_{\boldsymbol{k}}\protect\coloneqq F_{\mathrm{BZ}}(\boldsymbol{k})$,
$b_{k_{1},\cdots,k_{d}}\protect\coloneqq b_{k_{1}}\cdots b_{k_{d}}$. }
\tabularnewline
\toprule 
$\widetilde{\boldsymbol{k}}$  & $\widetilde{\boldsymbol{k}}^{\dagger}$  & $\varphi^{\pi}(\theta(\widetilde{\boldsymbol{k}}))$ \tabularnewline
\endfirsthead
\midrule 
\multicolumn{3}{c}{\tablename\ \thetable\ -- continued}\tabularnewline
\midrule 
$\widetilde{\boldsymbol{k}}$  & $\widetilde{\boldsymbol{k}}^{\dagger}$  & $\varphi^{\pi}(\theta(\widetilde{\boldsymbol{k}}))$ \tabularnewline
\midrule
\endhead
\midrule 
\multicolumn{3}{r}{Continued on next page}\tabularnewline
\endfoot
\endlastfoot
\midrule 
$\emptyset$  & $\emptyset$  & $F_{\emptyset}\otimes1$ \tabularnewline
$\bigl(\hat{1}\bigr)$  & $\bigl(\hat{1}\bigr)$  & $F_{\emptyset}\otimes b_{0}\allowbreak+F_{\emptyset}\otimes b_{1}$ \tabularnewline
$\bigl(\hat{2}\bigr)$  & $\bigl(1,\hat{1}\bigr)$  & $F_{\emptyset}\otimes b_{1}\allowbreak+F_{\emptyset}\otimes b_{2}$ \tabularnewline
$\bigl(2\bigr)$  & $\bigl(2\bigr)$  & $F_{2}\otimes1\allowbreak-F_{\emptyset}\otimes b_{1}\allowbreak-F_{\emptyset}\otimes b_{2}$ \tabularnewline
$\bigl(\hat{1},\hat{1}\bigr)$  & $\bigl(\hat{1},\hat{1}\bigr)$  & $F_{\emptyset}\otimes b_{1,0}\allowbreak+F_{\emptyset}\otimes b_{1,1}$ \tabularnewline
$\bigl(1,\hat{1}\bigr)$  & $\bigl(\hat{2}\bigr)$  & $F_{\emptyset}\otimes b_{1}\allowbreak+F_{\emptyset}\otimes b_{2}$ \tabularnewline
$\bigl(\hat{3}\bigr)$  & $\bigl(1,1,\hat{1}\bigr)$  & $F_{\emptyset}\otimes b_{2}\allowbreak+F_{\emptyset}\otimes b_{3}$ \tabularnewline
$\bigl(3\bigr)$  & $\bigl(1,2\bigr)$  & $F_{3}\otimes1\allowbreak-F_{\emptyset}\otimes b_{2}\allowbreak-F_{\emptyset}\otimes b_{3}$ \tabularnewline
$\bigl(\hat{1},\hat{2}\bigr)$  & $\bigl(1,\hat{1},\hat{1}\bigr)$  & $F_{\emptyset}\otimes b_{1,1}\allowbreak+F_{\emptyset}\otimes b_{1,2}$ \tabularnewline
$\bigl(\hat{1},2\bigr)$  & $\bigl(2,\hat{1}\bigr)$  & $F_{2}\otimes b_{0}\allowbreak+F_{2}\otimes b_{1}\allowbreak-F_{\emptyset}\otimes b_{1}\allowbreak-2\,F_{\emptyset}\otimes b_{2}\allowbreak-F_{\emptyset}\otimes b_{3}\allowbreak-F_{\emptyset}\otimes b_{1,0}\allowbreak-2\,F_{\emptyset}\otimes b_{1,1}\allowbreak-F_{\emptyset}\otimes b_{1,2}\allowbreak-F_{\emptyset}\otimes b_{2,0}\allowbreak-F_{\emptyset}\otimes b_{2,1}$ \tabularnewline
$\bigl(1,\hat{2}\bigr)$  & $\bigl(1,\hat{2}\bigr)$  & $-F_{2}\otimes b_{0}\allowbreak-F_{2}\otimes b_{1}\allowbreak+F_{\emptyset}\otimes b_{1}\allowbreak+3\,F_{\emptyset}\otimes b_{2}\allowbreak+2\,F_{\emptyset}\otimes b_{3}\allowbreak+F_{\emptyset}\otimes b_{1,0}\allowbreak+F_{\emptyset}\otimes b_{1,1}\allowbreak+F_{\emptyset}\otimes b_{2,0}\allowbreak+F_{\emptyset}\otimes b_{2,1}$ \tabularnewline
$\bigl(1,2\bigr)$  & $\bigl(3\bigr)$  & $F_{1,2}\otimes1\allowbreak-F_{\emptyset}\otimes b_{2}\allowbreak-F_{\emptyset}\otimes b_{3}$ \tabularnewline
$\bigl(\hat{2},\hat{1}\bigr)$  & $\bigl(\hat{1},1,\hat{1}\bigr)$  & $F_{\emptyset}\otimes b_{2,0}\allowbreak+F_{\emptyset}\otimes b_{2,1}$ \tabularnewline
$\bigl(2,\hat{1}\bigr)$  & $\bigl(\hat{1},2\bigr)$  & $F_{2}\otimes b_{0}\allowbreak+F_{2}\otimes b_{1}\allowbreak-F_{\emptyset}\otimes b_{1}\allowbreak-2\,F_{\emptyset}\otimes b_{2}\allowbreak-F_{\emptyset}\otimes b_{3}\allowbreak-F_{\emptyset}\otimes b_{1,0}\allowbreak-2\,F_{\emptyset}\otimes b_{1,1}\allowbreak-F_{\emptyset}\otimes b_{1,2}\allowbreak-F_{\emptyset}\otimes b_{2,0}\allowbreak-F_{\emptyset}\otimes b_{2,1}$ \tabularnewline
$\bigl(\hat{1},\hat{1},\hat{1}\bigr)$  & $\bigl(\hat{1},\hat{1},\hat{1}\bigr)$  & $F_{\emptyset}\otimes b_{1,1,0}\allowbreak+F_{\emptyset}\otimes b_{1,1,1}$ \tabularnewline
$\bigl(\hat{1},1,\hat{1}\bigr)$  & $\bigl(\hat{2},\hat{1}\bigr)$  & $F_{\emptyset}\otimes b_{2,0}\allowbreak+F_{\emptyset}\otimes b_{2,1}$ \tabularnewline
$\bigl(1,\hat{1},\hat{1}\bigr)$  & $\bigl(\hat{1},\hat{2}\bigr)$  & $F_{\emptyset}\otimes b_{1,1}\allowbreak+F_{\emptyset}\otimes b_{1,2}$ \tabularnewline
$\bigl(1,1,\hat{1}\bigr)$  & $\bigl(\hat{3}\bigr)$  & $F_{\emptyset}\otimes b_{2}\allowbreak+F_{\emptyset}\otimes b_{3}$ \tabularnewline
$\bigl(\hat{1},1,\hat{3}\bigr)$  & $\bigl(1,1,\hat{2},\hat{1}\bigr)$  & $-F_{3}\otimes b_{0,0}\allowbreak-F_{3}\otimes b_{0,1}\allowbreak-2\,F_{3}\otimes b_{1,0}\allowbreak-2\,F_{3}\otimes b_{1,1}\qquad\qquad\allowbreak-F_{2}\otimes b_{0,1}\allowbreak-F_{2}\otimes b_{0,2}\allowbreak-F_{2}\otimes b_{1,0}\allowbreak-3\,F_{2}\otimes b_{1,1}\allowbreak-2\,F_{2}\otimes b_{1,2}\allowbreak-F_{2}\otimes b_{2,0}\allowbreak-F_{2}\otimes b_{2,1}\allowbreak+F_{\emptyset}\otimes b_{1,1}\allowbreak+F_{\emptyset}\otimes b_{1,2}\allowbreak+2\,F_{\emptyset}\otimes b_{2,0}\allowbreak+6\,F_{\emptyset}\otimes b_{2,1}\allowbreak+5\,F_{\emptyset}\otimes b_{2,2}\allowbreak+F_{\emptyset}\otimes b_{2,3}\allowbreak+6\,F_{\emptyset}\otimes b_{3,0}\allowbreak+9\,F_{\emptyset}\otimes b_{3,1}\allowbreak+3\,F_{\emptyset}\otimes b_{3,2}\allowbreak+4\,F_{\emptyset}\otimes b_{4,0}\allowbreak+4\,F_{\emptyset}\otimes b_{4,1}\allowbreak+F_{\emptyset}\otimes b_{1,0,1}\allowbreak+F_{\emptyset}\otimes b_{1,0,2}\allowbreak+F_{\emptyset}\otimes b_{1,1,0}\allowbreak+4\,F_{\emptyset}\otimes b_{1,1,1}\allowbreak+3\,F_{\emptyset}\otimes b_{1,1,2}\allowbreak+2\,F_{\emptyset}\otimes b_{1,2,0}\allowbreak+3\,F_{\emptyset}\otimes b_{1,2,1}\allowbreak+F_{\emptyset}\otimes b_{1,2,2}\allowbreak+F_{\emptyset}\otimes b_{1,3,0}\allowbreak+F_{\emptyset}\otimes b_{1,3,1}\allowbreak+F_{\emptyset}\otimes b_{2,0,0}\allowbreak+2\,F_{\emptyset}\otimes b_{2,0,1}\allowbreak+F_{\emptyset}\otimes b_{2,0,2}\allowbreak+3\,F_{\emptyset}\otimes b_{2,1,0}\allowbreak+5\,F_{\emptyset}\otimes b_{2,1,1}\allowbreak+2\,F_{\emptyset}\otimes b_{2,1,2}\allowbreak+F_{\emptyset}\otimes b_{2,2,0}\allowbreak+F_{\emptyset}\otimes b_{2,2,1}\allowbreak+F_{\emptyset}\otimes b_{3,0,0}\allowbreak+F_{\emptyset}\otimes b_{3,0,1}\allowbreak+2\,F_{\emptyset}\otimes b_{3,1,0}\allowbreak+2\,F_{\emptyset}\otimes b_{3,1,1}$ \tabularnewline
$\bigl(1,1,\hat{2},\hat{1}\bigr)$  & $\bigl(\hat{1},1,\hat{3}\bigr)$  & $-F_{1,2}\otimes b_{0,0}\allowbreak-F_{1,2}\otimes b_{0,1}\allowbreak-2\,F_{1,2}\otimes b_{1,0}\allowbreak-2\,F_{1,2}\otimes b_{1,1}\allowbreak-F_{2}\otimes b_{0,1}\allowbreak-F_{2}\otimes b_{0,2}\allowbreak-F_{2}\otimes b_{1,0}\allowbreak-3\,F_{2}\otimes b_{1,1}\allowbreak-2\,F_{2}\otimes b_{1,2}\allowbreak-F_{2}\otimes b_{2,0}\allowbreak-F_{2}\otimes b_{2,1}\allowbreak+F_{\emptyset}\otimes b_{1,1}\allowbreak+F_{\emptyset}\otimes b_{1,2}\allowbreak+2\,F_{\emptyset}\otimes b_{2,0}\allowbreak+6\,F_{\emptyset}\otimes b_{2,1}\allowbreak+5\,F_{\emptyset}\otimes b_{2,2}\allowbreak+F_{\emptyset}\otimes b_{2,3}\allowbreak+6\,F_{\emptyset}\otimes b_{3,0}\allowbreak+9\,F_{\emptyset}\otimes b_{3,1}\allowbreak+3\,F_{\emptyset}\otimes b_{3,2}\allowbreak+4\,F_{\emptyset}\otimes b_{4,0}\allowbreak+4\,F_{\emptyset}\otimes b_{4,1}\allowbreak+F_{\emptyset}\otimes b_{1,0,1}\allowbreak+F_{\emptyset}\otimes b_{1,0,2}\allowbreak+F_{\emptyset}\otimes b_{1,1,0}\allowbreak+4\,F_{\emptyset}\otimes b_{1,1,1}\allowbreak+3\,F_{\emptyset}\otimes b_{1,1,2}\allowbreak+2\,F_{\emptyset}\otimes b_{1,2,0}\allowbreak+3\,F_{\emptyset}\otimes b_{1,2,1}\allowbreak+F_{\emptyset}\otimes b_{1,2,2}\allowbreak+F_{\emptyset}\otimes b_{1,3,0}\allowbreak+F_{\emptyset}\otimes b_{1,3,1}\allowbreak+F_{\emptyset}\otimes b_{2,0,0}\allowbreak+2\,F_{\emptyset}\otimes b_{2,0,1}\allowbreak+F_{\emptyset}\otimes b_{2,0,2}\allowbreak+3\,F_{\emptyset}\otimes b_{2,1,0}\allowbreak+5\,F_{\emptyset}\otimes b_{2,1,1}\allowbreak+2\,F_{\emptyset}\otimes b_{2,1,2}\allowbreak+F_{\emptyset}\otimes b_{2,2,0}\allowbreak+F_{\emptyset}\otimes b_{2,2,1}\allowbreak+F_{\emptyset}\otimes b_{3,0,0}\allowbreak+F_{\emptyset}\otimes b_{3,0,1}\allowbreak+2\,F_{\emptyset}\otimes b_{3,1,0}\allowbreak+2\,F_{\emptyset}\otimes b_{3,1,1}$ \tabularnewline
$\bigl(2,4\bigr)$  & $\bigl(1,1,2,2\bigr)$  & $F_{2,4}\otimes1\allowbreak-F_{4}\otimes b_{0}\allowbreak-3\,F_{4}\otimes b_{1}\allowbreak-2\,F_{4}\otimes b_{2}\allowbreak-2\,F_{3}\otimes b_{1}\allowbreak-4\,F_{3}\otimes b_{2}\allowbreak-2\,F_{3}\otimes b_{3}\allowbreak-3\,F_{2}\otimes b_{2}\allowbreak-7\,F_{2}\otimes b_{3}\allowbreak-4\,F_{2}\otimes b_{4}\allowbreak+6\,F_{\emptyset}\otimes b_{3}\allowbreak+23\,F_{\emptyset}\otimes b_{4}\allowbreak+28\,F_{\emptyset}\otimes b_{5}\allowbreak+11\,F_{\emptyset}\otimes b_{6}\allowbreak+3\,F_{\emptyset}\otimes b_{1,2}\allowbreak+10\,F_{\emptyset}\otimes b_{1,3}\allowbreak+11\,F_{\emptyset}\otimes b_{1,4}\allowbreak+4\,F_{\emptyset}\otimes b_{1,5}\allowbreak+2\,F_{\emptyset}\otimes b_{2,1}\allowbreak+7\,F_{\emptyset}\otimes b_{2,2}\allowbreak+9\,F_{\emptyset}\otimes b_{2,3}\allowbreak+4\,F_{\emptyset}\otimes b_{2,4}\allowbreak+F_{\emptyset}\otimes b_{3,0}\allowbreak+5\,F_{\emptyset}\otimes b_{3,1}\allowbreak+6\,F_{\emptyset}\otimes b_{3,2}\allowbreak+2\,F_{\emptyset}\otimes b_{3,3}\allowbreak+F_{\emptyset}\otimes b_{4,0}\allowbreak+3\,F_{\emptyset}\otimes b_{4,1}\allowbreak+2\,F_{\emptyset}\otimes b_{4,2}$ \tabularnewline
$\bigl(1,1,2,2\bigr)$  & $\bigl(2,4\bigr)$  & $F_{1,1,2,2}\otimes1\allowbreak-F_{1,1,2}\otimes b_{0}\allowbreak-3\,F_{1,1,2}\otimes b_{1}\allowbreak-2\,F_{1,1,2}\otimes b_{2}\allowbreak-2\,F_{1,2}\otimes b_{1}\allowbreak-4\,F_{1,2}\otimes b_{2}\allowbreak-2\,F_{1,2}\otimes b_{3}\allowbreak-3\,F_{2}\otimes b_{2}\allowbreak-7\,F_{2}\otimes b_{3}\allowbreak-4\,F_{2}\otimes b_{4}\allowbreak+6\,F_{\emptyset}\otimes b_{3}\allowbreak+23\,F_{\emptyset}\otimes b_{4}\allowbreak+28\,F_{\emptyset}\otimes b_{5}\allowbreak+11\,F_{\emptyset}\otimes b_{6}\allowbreak+3\,F_{\emptyset}\otimes b_{1,2}\allowbreak+10\,F_{\emptyset}\otimes b_{1,3}\allowbreak+11\,F_{\emptyset}\otimes b_{1,4}\allowbreak+4\,F_{\emptyset}\otimes b_{1,5}\allowbreak+2\,F_{\emptyset}\otimes b_{2,1}\allowbreak+7\,F_{\emptyset}\otimes b_{2,2}\allowbreak+9\,F_{\emptyset}\otimes b_{2,3}\allowbreak+4\,F_{\emptyset}\otimes b_{2,4}\allowbreak+F_{\emptyset}\otimes b_{3,0}\allowbreak+5\,F_{\emptyset}\otimes b_{3,1}\allowbreak+6\,F_{\emptyset}\otimes b_{3,2}\allowbreak+2\,F_{\emptyset}\otimes b_{3,3}\allowbreak+F_{\emptyset}\otimes b_{4,0}\allowbreak+3\,F_{\emptyset}\otimes b_{4,1}\allowbreak+2\,F_{\emptyset}\otimes b_{4,2}$ \tabularnewline
\end{longtable}

\end{document}